\documentclass[reqno,english]{amsart}%
\usepackage{amsfonts,amsmath,latexsym,verbatim,amscd,mathrsfs,color,array}
\usepackage[colorlinks=true]{hyperref}
\usepackage{amsmath,amssymb,amsthm,amsfonts,graphicx,color}
\usepackage{amssymb}
\usepackage{pdfsync}
\usepackage{epstopdf}
\usepackage{cite}
\usepackage{graphicx}
\usepackage{amsmath}
\usepackage{refcheck}
\usepackage{amsfonts}%
\providecommand{\U}[1]{\protect\rule{.1in}{.1in}}

\numberwithin{equation}{section}

\newtheorem{theorem}{Theorem}[section]
\newtheorem{corollary}{Corollary}[section]
\newtheorem{lemma}{Lemma}[section]
\newtheorem{proposition}{Proposition}[section]
\newtheorem{remark}{Remark}[section]
\newtheorem{definition}{Definition}[section]

\numberwithin{equation}{section}

\newcommand{\bbr}{\mathbb{R}}

\newcommand{\bbn}{\mathbb{N}}
\newcommand{\ve}{\varepsilon}

\newcommand{\bd}{\begin{definition}}
\newcommand{\ed}{\end{definition}}

\newcommand{\br}{\begin{remark}}
\newcommand{\er}{\end{remark}}

\newcommand{\be}{\begin{equation}}
\newcommand{\ee}{\end{equation}}

\newcommand{\bc}{\begin{corollary}}
\newcommand{\ec}{\end{corollary}}

\begin{document}

\title[Caffarelli-Kohn-Nirenberg inequality]{ Stability  of Caffarelli-Kohn-Nirenberg inequality}

\author[J. Wei]{Juncheng Wei}
\address{\noindent Department of Mathematics, University of British Columbia,
Vancouver, B.C., Canada, V6T 1Z2}
\email{jcwei@math.ubc.ca}

\author[Y.Wu]{Yuanze Wu}
\address{\noindent  School of Mathematics, China
University of Mining and Technology, Xuzhou, 221116, P.R. China }
\email{wuyz850306@cumt.edu.cn}

\begin{abstract}
In this paper, we consider the Caffarelli-Kohn-Nirenberg (CKN) inequality:
\begin{eqnarray*}
\bigg(\int_{\bbr^N}|x|^{-b(p+1)}|u|^{p+1}dx\bigg)^{\frac{2}{p+1}}\leq C_{a,b,N}\int_{\bbr^N}|x|^{-2a}|\nabla u|^2dx
\end{eqnarray*}
where $N\geq3$, $a<\frac{N-2}{2}$, $a\leq b\leq a+1$ and $p=\frac{N+2(1+a-b)}{N-2(1+a-b)}$.   It is well-known that up to dilations $\tau^{\frac{N-2}{2}-a}u(\tau x)$ and scalar multiplications $Cu(x)$, the CKN inequality has a unique extremal function $W(x)$ which is positive and radially symmetric
in the parameter region $b_{FS}(a)\leq b<a+1$ with $a<0$ and $a\leq b<a+1$ with $a\geq0$ and $a+b>0$, where $b_{FS}(a)$ is the Felli-Schneider curve.  We prove that in the above parameter region the following stabilities hold:
\begin{enumerate}
\item[$(1)$] \quad  stability of CKN inequality in the functional inequality setting
$$dist_{D^{1,2}_{a}}^2(u, \mathcal{Z})\lesssim\|u\|^2_{D^{1,2}_a(\bbr^N)}-C_{a,b,N}^{-1}\|u\|^2_{L^{p+1}(|x|^{-b(p+1)},\bbr^N)}$$
where $\mathcal{Z}= \{ c W_\tau\mid c\in\bbr\backslash\{0\}, \tau>0\}$;
\item[$(2)$]\quad   stability of CKN inequality in the critical point setting (in the class of nonnegative functions)
\begin{eqnarray*}
dist_{D_a^{1,2}}(u, \mathcal{Z}_0^\nu)\lesssim\left\{\aligned &\Gamma(u),\quad p>2\text{ or }\nu=1,\\
&\Gamma(u)|\log\Gamma(u)|^{\frac12},\quad p=2\text{ and }\nu\geq2,\\
&\Gamma(u)^{\frac{p}{2}},\quad 1<p<2\text{ and }\nu\geq2,
\endaligned\right.
\end{eqnarray*}
where $\Gamma (u)=\|div(|x|^{-a}\nabla u)+|x|^{-b(p+1)}|u|^{p-1}u\|_{(D^{1,2}_a)^{'}}$ and
$$\mathcal{Z}_0^\nu=\{(W_{\tau_1},W_{\tau_2},\cdots,W_{\tau_\nu})\mid \tau_i>0\}.$$
\end{enumerate}

Our results generalize the recent works in \cite{CFM2017,FG2021,DSW2021} on the sharp stability of profile decompositions for the the special case $a=b=0$ (the Sobolev inequality) to the above parameter region of the CKN inequality.  This parameter region is optimal for such stabilities in the sense that in the region $a<b<b_{FS}(a)$ with $a<0$, the nonnegative solution of the Euler-Lagrange equation of CKN inequality is no longer unique.

\vspace{3mm} \noindent{\bf Keywords:} Caffarelli-Kohn-Nirenberg inequality; Sharp stability; Profile decomposition.

\vspace{3mm}\noindent {\bf AMS} Subject Classification 2010: 35B09; 35B33; 35B40; 35J20.%

\end{abstract}

\date{}
\maketitle

\section{Introduction}
In this paper, we consider the following Caffarelli-Kohn-Nirenberg (CKN for short) inequality:
\begin{eqnarray}\label{eq0001}
\bigg(\int_{\bbr^N}|x|^{-b(p+1)}|u|^{p+1}dx\bigg)^{\frac{2}{p+1}}\leq C_{a,b,N}\int_{\bbr^N}|x|^{-2a}|\nabla u|^2dx,
\end{eqnarray}
where
\begin{eqnarray}\label{eq0003}
N\geq3,\quad -\infty<a<\frac{N-2}{2},\quad a\leq b\leq a+1\quad \text{and} \quad p=\frac{N+2(1+a-b)}{N-2(1+a-b)},
\end{eqnarray}
$u\in D^{1,2}_{a}(\bbr^N)$ and $D^{1,2}_{a}(\bbr^N)$ is the Hilbert space given by
\begin{eqnarray}\label{eqn886}
D^{1,2}_{a}(\bbr^N)=\{u\in D^{1,2}(\bbr^N)\mid \int_{\bbr^N}|x|^{-2a}|\nabla u|^2dx<+\infty\}
\end{eqnarray}
with the inner product
\begin{eqnarray*}
\langle u,v \rangle_{D^{1,2}_{a}(\bbr^N)}=\int_{\bbr^N}|x|^{-2a}\nabla u\nabla vdx
\end{eqnarray*}
and $D^{1,2}(\bbr^N)=\dot{W}^{1,2}(\bbr^N)$ being the usual homogeneous Sobolev space (cf. \cite[Definition~2.1]{FG2021}).

\vskip0.2in

\eqref{eq0001} is established in the celebrated paper \cite{CKN1984} by Caffarelli, Kohn and Nirenberg, as it is named, for a much more general version.  Moreover, as pointed out in \cite{CW2001}, a fundamental task in understanding the CKN inequality~\eqref{eq0001} is to study the best constants, existence (and nonexistence) of extremal functions, as well as their qualitative properties for parameters $a$ and $b$ in the full region~\eqref{eq0003}, since \eqref{eq0001} contains the classical Sobolev inequality ($a=b=0$) and the classical Hardy inequality ($a=0, b=1$) as special cases, which have played important roles in many applications by virtue of the complete knowledge about the best constants, extremal functions, and their qualitative properties.

\vskip0.2in

Under the condition~\eqref{eq0003}, it is well-known (cf. \cite{A1976,CC1993,CW2001,L1983,T1976}) that \eqref{eq0001} has extremal functions if and only if either for $a<b<a+1$ with $a<0$ or for $a\leq b<a+1$ with $a\geq0$.  Moreover, let
\begin{eqnarray}\label{eqn991}
b_{FS}(a)=\frac{N(a_c-a)}{2\sqrt{(a_c-a)^2+(N-1)}}+a-a_c>a
\end{eqnarray}
be the Felli-Schneider curve found in \cite{FS2003},
then it is also well-known (cf. \cite{A1976,CC1993,DELT2009,DEL2012,DEL2016,FS2003,L1983,T1976}) that up to dilations $\tau^{a_c-a}u(\tau x)$ and scalar multiplications $Cu(x)$ (also up to translations $u(x+y)$ in the special case $a=b=0$), \eqref{eq0001} has a unique extremal function
\begin{eqnarray}\label{eq0004}
W(x)=(2(p+1)(a_c-a)^2)^{\frac{1}{(p-1)}}\bigg(1+|x|^{(a_c-a)(p-1)}\bigg)^{-\frac{2}{p-1}}
\end{eqnarray}
either for $b_{FS}(a)\leq b<a+1$ with $a<0$ or for $a\leq b<a+1$ with $a\geq0$ while, extremal functions of \eqref{eq0001} must be non-radial for $a<b<b_{FS}(a)$ with $a<0$.  Here, for the sake of simplicity, we denote $a_c=\frac{N-2}{2}$, as that in \cite{DELT2009,DEL2012,DEL2016}.  On the other hand, it has been proved in \cite{LW2004} that extremal functions of \eqref{eq0001} must have $\mathcal{O}(N-1)$ symmetry for $a<b<b_{FS}(a)$ with $a<0$, that is, extremal functions of \eqref{eq0001} must depend on the radius $r$ and the angle $\theta_{N}$ between the positive $x_N$-axis and $\overrightarrow{Ox}$ for $a<b<b_{FS}(a)$ with $a<0$ up to rotations.  To our best knowledge, whether the extremal function of \eqref{eq0001} is unique or not for $a<b<b_{FS}(a)$ with $a<0$ is still unknown.

\vskip0.2in

Once the extremal functions of \eqref{eq0001} are well understood, it is natural to study quantitative stability for the CKN inequality~\eqref{eq0001} by asking whether the deviation of a given function from attaining equality in \eqref{eq0001} controls its distance to the family of extremal functions.  These studies were initialed by Brez\'is and Lieb in \cite{BL1985} by raising an open question for the classical Sobolev inequality ($a=b=0$) which was settled by Bianchi and Egnell in \cite{BE1991} while, quantitative stability for the Hardy-Sobolev inequality ($a=0$, $0<b<1$) was studied in \cite{RSW2002}.
Since the extremal function of \eqref{eq0001} is unique up to dilations $\tau^{a_c-a}u(\tau x)$ and scalar multiplications $Cu(x)$ either for $b_{FS}(a)\leq b<a+1$ with $a<0$ or for $a\leq b<a+1$ with $a\geq0$ and $a+b>0$,
the smooth manifold
\begin{eqnarray*}
\mathcal{Z}=\{cW_\tau(x)\mid c\in\bbr\backslash\{0\}\text{ and }\tau>0\}
\end{eqnarray*}
is all extremal functions of \eqref{eq0001} in the above parameter region.  Thus, it is natural to extend the quantitative stability for the Sobolev inequality ($a=b=0$) and the Hardy-Sobolev inequality ($a=0$, $0<b<1$) to the CKN inequality~\eqref{eq0001} in the above parameter region.  Our first result in this paper, which devoted to this aspect, can be stated as follows.
\begin{theorem}\label{thm0001}
Let $b_{FS}(a)$
be the Felli-Schneider curve given by \eqref{eqn991} and assume that either
\begin{enumerate}
\item[$(1)$]\quad $b_{FS}(a)\leq b<a+1$ with $a<0$ or
\item[$(2)$]\quad $a\leq b<a+1$ with $a\geq0$ and $a+b>0$.
\end{enumerate}
Then
\begin{eqnarray*}
dist_{D^{1,2}_{a}}^2(u, \mathcal{Z})\lesssim\|u\|^2_{D^{1,2}_a(\bbr^N)}-C_{a,b,N}^{-1}\|u\|^2_{L^{p+1}(|x|^{-b(p+1)},\bbr^N)}
\end{eqnarray*}
for all $u\in D^{1,2}_a(\bbr^N)$, where $L^{p+1}(|x|^{-b(p+1)},\bbr^N)$ is the usual weighted Lebesgue space and its usual norm is given by
\begin{eqnarray*}
\|u\|_{L^{p+1}(|x|^{-b(p+1)},\bbr^N)}=\bigg(\int_{\bbr^N}|x|^{-b(p+1)}|u|^{p+1}dx\bigg)^{\frac{1}{p+1}}.
\end{eqnarray*}
\end{theorem}

\begin{remark}
The conditions $b_{FS}(a)\leq b<a+1$ with $a<0$ and $a\leq b<a+1$ with $a\geq0$ and $a+b>0$ in Theorem~\ref{thm0001} is sharp in the sense that, for $a<b<b_{FS}(a)$ with $a<0$, the extremal function of \eqref{eq0001} is no longer $cW_\tau$.
\end{remark}

\vskip0.2in

On the other hand, it is well-known that the Euler-Lagrange equation of the Sobolev inequality ($a=b=0$) is given by
\begin{eqnarray}\label{eqn880}
-\Delta u=|u|^{\frac{4}{N-2}}u, \quad\text{in }\bbr^N
\end{eqnarray}
and the Aubin-Talanti bubbles, given by
\begin{eqnarray*}
U[z,\lambda]=[N(N-2)]^{\frac{N-2}{4}}\bigg(\frac{\lambda}{\lambda^2+|x-z|^2}\bigg)^{\frac{N-2}{2}},
\end{eqnarray*}
are the only poaitive solutions of \eqref{eqn880}, where $z\in\bbr^N$ and $\lambda>0$.  Thus, the smooth manifold (except $c=0$)
\begin{eqnarray*}
\mathcal{M}=\{cU[z,\lambda]\mid c\in\bbr, z\in\bbr^N, \lambda>0\}
\end{eqnarray*}
is all nonnagetive solutions of \eqref{eqn880}. Moreover,
Struwe proved in \cite{S1984} the following well-known stability of profile decompositions to \eqref{eqn880} for nonnegative functions.
\begin{theorem}\label{thm0003}
Let $N\geq3$ and $\nu\geq1$ be positive integers.  Let $\{u_n\}\subset D^{1,2}(\bbr^N)$ be a nonnegative sequence with
\begin{eqnarray*}
(\nu-\frac12)S^\frac{N}{2}<\|u_n\|_{D^{1,2}(\bbr^N)}^2<(\nu+\frac12)S^\frac{N}{2},
\end{eqnarray*}
where $S$ is the best Sobolev constant.
Assume that $\|\Delta u_n+|u_n|^{\frac{4}{N-2}}u_n\|_{H^{-1}}\to0$ as $n\to\infty$, then there exist a sequence $(z_1^{(n)}, z_2^{(n)}, \cdots, z_\nu^{(n)})$ of $\nu$-tuples of points in $\bbr^N$ and a sequence of $(\lambda_1^{(n)}, \lambda_2^{(n)}, \cdots, \lambda_\nu^{(n)})$ of $\nu$-tuples of positive real numbers such that
\begin{eqnarray*}
\|\nabla u_n-\sum_{i=1}^{\nu}\nabla U[z_i^{(n)}, \lambda_i^{(n)}]\|_{L^2(\bbr^N)}\to0\quad\text{as }n\to\infty.
\end{eqnarray*}
\end{theorem}
In the recent papers \cite{CFM2017,FG2021}, Figalli et al. initialed a study on the quantitative version of Theorem~\ref{thm0003} by proving
\begin{enumerate}
\item[$(1)$] (Ciraolo-Figalli-Maggi \cite{CFM2017})\quad Let $N\geq3$ and $u\in D^{1,2}(\bbr^N)$ be positive such that $\|\nabla u\|_{L^2(\bbr^N)}^2\leq\frac{3}{2}S^{\frac{N}{2}}$ and $\|\Delta u+|u|^{\frac{4}{N-2}}u\|_{H^{-1}}\leq\delta$ for some $\delta>0$ sufficiently small, then
    \begin{eqnarray*}
    dist_{D^{1,2}}(u, \mathcal{M}_0)\lesssim\|\Delta u+|u|^{\frac{4}{N-2}}u\|_{H^{-1}},
    \end{eqnarray*}
    where $\mathcal{M}_0=\{U[z, \lambda]\mid z\in\bbr^N, \lambda>0\}.$
\item[$(2)$] (Figalli-Glaudo \cite{FG2021})\quad Let $u\in D^{1,2}(\bbr^N)$ be nonnegative such that
\begin{eqnarray*}
(\nu-\frac12)S^\frac{N}{2}<\|u\|_{D^{1,2}(\bbr^N)}^2<(\nu+\frac12)S^\frac{N}{2}
\end{eqnarray*}
and $\|\Delta u+|u|^{\frac{4}{N-2}}u\|_{H^{-1}}\leq\delta$ for some $\delta>0$ sufficiently small,
then for $3\leq N\leq5$,
    \begin{eqnarray*}
    dist_{D^{1,2}}(u, \mathcal{M}_0^\nu)\lesssim\|\Delta u+|u|^{\frac{4}{N-2}}u\|_{H^{-1}}
    \end{eqnarray*}
    where
    \begin{eqnarray*}
    \mathcal{M}_0^\nu=\{\sum_{i=1}^\nu U[z_i, \lambda_i] \mid z_i\in\bbr^N, \lambda_i>0\}.
    \end{eqnarray*}
\end{enumerate}
\begin{remark}
The stability obtained in \cite{FG2021} is more general than the conclusion~$(2)$ stated here in the sense that, $u$ could be sign-changing if $u$ is close to the sum of $U[z_i, \lambda_i]$ in $D^{1,2}(\bbr^N)$ where $U[z_i, \lambda_i]$ are weakly interacting (the definition of weakly interaction can be found in \cite[Definition~3.1]{FG2021}).  We choose to state the conclusion~$(2)$ since it is more close to Struwe's theorem on the stability of profile decompositions to \eqref{eqn880} for nonnegative functions.
\end{remark}
As pointed out in \cite{FG2021}, it is rather surprisingly that the conclusion~$(2)$ is false for $N\geq6$.  Figalli and Glaudo constructed a counterexample of the conclusion~$(2)$ for $N\geq6$ with two bubbles and conjectured in \cite{FG2021} that the quantitative version of Theorem~\ref{thm0003} for $N\geq6$ and $\nu>1$ will be
\begin{eqnarray*}
dist_{D^{1,2}}(u, \mathcal{M}_0^\nu)\lesssim\left\{\aligned&\|\Delta u+|u|u\|_{H^{-1}}|\ln(\|\Delta u+|u|u\|_{H^{-1}})|,\quad N=6;\\
&\|\Delta u+|u|^{\frac{4}{N-2}}u\|_{H^{-1}}^{\gamma(N)},\quad N\geq7
\endaligned\right.
\end{eqnarray*}
with $0<\gamma(N)<1$ under the same assumptions of the conclusion~$(2)$.  In the very recent work \cite{DSW2021}, the first author of the current paper, together with Deng and Sun, proved that the quantitative version of Theorem~\ref{thm0003} for $N\geq6$ and $\nu>1$ is actually
\begin{eqnarray*}
dist_{D^{1,2}}(u, \mathcal{M}_0^\nu)\lesssim\left\{\aligned&\|\Delta u+|u|u\|_{H^{-1}}|\ln(\|\Delta u+|u|u\|_{H^{-1}})|^{\frac12},\ \ N=6;\\
&\|\Delta u+|u|^{\frac{4}{N-2}}u\|_{H^{-1}}^{\frac{N+2}{2(N-2)}},\ \ N\geq7
\endaligned\right.
\end{eqnarray*}
under the same assumptions of the conclusion~$(2)$, where the orders of the right hand sides in above estimates are optimal.  We would like to refer the recent works \cite{CFM2017,FG2021} once more for more motivations, discussions and applications of the study on the quantitative version of Theorem~\ref{thm0003}.

\vskip0.2in

Since the Sobolev inequality ($a=b=0$) is a special case of the CKN inequality~\eqref{eq0001} and the Euler-Lagrange equation of \eqref{eq0001} is given by
\begin{eqnarray}\label{eq0018}
-div(|x|^{-a}\nabla u)=|x|^{-b(p+1)}|u|^{p-1}u, \quad\text{in }\bbr^N,
\end{eqnarray}
it is natural to ask whether the stability of profile decompositions to \eqref{eq0018} for nonnegative functions which is similar to that of \eqref{eqn880} holds or not.  Our second main result of this paper, which devoted to this natural question, can be stated as follows.
\begin{theorem}\label{thm0002}
Let $N\geq3$ and $\nu\geq1$ be positive integers.  Let $b_{FS}(a)$
be the Felli-Schneider curve given by \eqref{eqn991} and assume that either $b_{FS}(a)\leq b<a+1$ with $a<0$ or $a\leq b<a+1$ with $a\geq0$ and $a+b>0$.
Then for any nonnegative $u\in D^{1,2}_a(\bbr^N)$ such that
\begin{eqnarray}\label{eqn99}
(\nu-\frac12)(C_{a,b,N}^{-1})^{\frac{p+1}{p-1}}<\|u\|_{D^{1,2}_a(\bbr^N)}^2<(\nu+\frac12)(C_{a,b,N}^{-1})^{\frac{p+1}{p-1}}
\end{eqnarray}
and $\Gamma(u)\leq\delta$ with some $\delta>0$ sufficiently small,
we have
\begin{eqnarray*}
dist_{D_a^{1,2}}(u, \mathcal{Z}_0^\nu)\lesssim\left\{\aligned &\Gamma(u),\quad p>2\text{ or }\nu=1,\\
&\Gamma(u)|\log\Gamma(u)|^{\frac12},\quad p=2\text{ and }\nu\geq2,\\
&\Gamma(u)^{\frac{p}{2}},\quad 1<p<2\text{ and }\nu\geq2.
\endaligned\right.
\end{eqnarray*}
where $\Gamma (u)=\|div(|x|^{-a}\nabla u)+|x|^{-b(p+1)}|u|^{p-1}u\|_{(D^{1,2}_a)^{'}}$
and
$$
\mathcal{Z}_0^\nu=\{\sum_{j=1}^\nu W_{\tau_j} \mid \tau_j>0\}.
$$
Moreover, this stability is sharp in the sense that, there exists nonnegative $u_*\in D^{1,2}_a(\bbr^N)$, satisfying \eqref{eqn99} and $\Gamma(u_*)\leq\delta$ for some $\delta>0$ sufficiently small, such that
\begin{eqnarray*}
dist_{D_a^{1,2}}(u_*, \mathcal{Z}_0^\nu)\gtrsim\left\{\aligned &\Gamma(u_*),\quad p>2\text{ or }\nu=1,\\
&\Gamma(u_*)|\log\Gamma(u_*)|^{\frac12},\quad p=2\text{ and }\nu\geq2,\\
&\Gamma(u_*)^{\frac{p}{2}},\quad 1<p<2\text{ and }\nu\geq2.
\endaligned\right.
\end{eqnarray*}
\end{theorem}

\begin{remark}
\begin{enumerate}
\item[$(1)$]\quad The conditions $b_{FS}(a)\leq b<a+1$ with $a<0$ and $a\leq b<a+1$ with $a\geq0$ and $a+b>0$ in Theorem~\ref{thm0002} is optimal in the sense that, for $a<b<b_{FS}(a)$ with $a<0$, the nonnegative solutions of \eqref{eq0018} is no longer unique.
\item[$(2)$]\quad $(2)$ of Theorem~\ref{thm0002} can be generalized to more general class of functions as that in \cite[Theorem~3.3]{FG2021} by introducing the concept of weakly interaction for the nonnegative solutions of \eqref{eq0018} as that in \cite[Definition~3.1]{FG2021}.
\end{enumerate}
\end{remark}

\noindent{\bf\large Notations.} Throughout this paper, $C$ and $C'$ are indiscriminately used to denote various absolutely positive constants.  $a\sim b$ means that $C'b\leq a\leq Cb$ and $a\lesssim b$ means that $a\leq Cb$.

\section{Preliminaries}
Let $D^{1,2}_{a}(\bbr^N)$ be the Hilbert space given by \eqref{eqn886} with the norm $\|\cdot\|_{D^{1,2}_a(\bbr^N)}$.
By \cite[Proposition~2.2]{CW2001}, $D^{1,2}_{a}(\bbr^N)$ is isomorphic to the Hilbert space $H^1(\mathcal{C})$ by the transformation
\begin{eqnarray}\label{eq0007}
u(x)=|x|^{-(a_c-a)}v(-\ln|x|,\frac{x}{|x|}),
\end{eqnarray}
where $\mathcal{C}=\bbr\times \mathbb{S}^{N-1}$ is the standard cylinder, the inner product in $H^1(\mathcal{C})$ is given by
\begin{eqnarray*}
\langle w,v \rangle_{H^1(\mathcal{C})}=\int_{\mathcal{C}}\nabla w\nabla v+(a_c-a)^2 uv d\mu
\end{eqnarray*}
with $d\mu$ being the volume element on $\mathcal{C}$,
$u\in D^{1,2}_{a}(\bbr^N)$ and $w,v\in H^1(\mathcal{C})$.

\vskip0.12in

The CKN inequality~\eqref{eq0001} can be rewritten as the following minimizing problem:
\begin{eqnarray}\label{eq0002}
C_{a,b,N}^{-1}=\inf_{u\in D^{1,2}_{a}(\bbr^N)\backslash\{0\}}\frac{\|u\|^2_{D^{1,2}_a(\bbr^N)}}{\|u\|^2_{L^{p+1}(|x|^{-b(p+1)},\bbr^N)}},
\end{eqnarray}
where $L^{p+1}(|x|^{-b(p+1)},\bbr^N)$ is the usual weighted Lebesgue space and its usual norm is given by $\|u\|_{L^{p+1}(|x|^{-b(p+1)},\bbr^N)}=\bigg(\int_{\bbr^N}|x|^{-b(p+1)}|u|^{p+1}dx\bigg)^{\frac{1}{p+1}}$.
The Euler-Lagrange equation of the minimizing problem~\eqref{eq0002} is given by \eqref{eq0018}.
By \eqref{eq0007}, \eqref{eq0002} is equivalent to the following minimizing problem:
\begin{eqnarray}\label{eq0009}
C_{a,b,N}^{-1}=\inf_{v\in H^1(\mathcal{C})\backslash\{0\}}\frac{\|v\|^2_{H^{1}(\mathcal{C})}}{\|u\|^2_{L^{p+1}(\mathcal{C})}},
\end{eqnarray}
where $\|\cdot\|_{L^{p+1}(\mathcal{C})}$ is the usual norm in the Lebesgue space $L^{p+1}(\mathcal{C})$.  Let $t=-\ln|x|$ and $\theta=\frac{x}{|x|}$ for $x\in\bbr^N\backslash\{0\}$, then by \cite[Proposition~2.2]{CW2001}, \eqref{eq0018} is equivalent to the following equation of $v$:
\begin{eqnarray}\label{eq0006}
-\Delta_{\theta}v-\partial_t^2v+(a_c-a)^2v=|v|^{p-1}v, \quad\text{in }\mathcal{C}
\end{eqnarray}
where $\Delta_{\theta}$ is the Laplace-Beltrami operator on $\mathbb{S}^{N-1}$.

\vskip0.12in

Clearly, minimizers of \eqref{eq0002} are ground states of \eqref{eq0018}.  It is also well-known (cf. \cite{CC1993,CW2001,FS2003}) that up to dilations $u_{\tau}(x)=\tau^{a_c-a}u(\tau x)$ and scalar multiplications $Cu(x)$ (also up to translations $u(x+y)$ for the spacial case $a=b=0$), the radial function $W(x)$ given by \eqref{eq0004} is also the unique minimizer of the following minimizing problem
\begin{eqnarray*}\label{eq0005}
C_{a,b,N,rad}^{-1}=\inf_{u\in D^{1,2}_{a,rad}(\bbr^N)\backslash\{0\}}\frac{\|u\|^2_{D^{1,2}_a(\bbr^N)}}{\|u\|^2_{L^{p+1}(|x|^{-b(p+1)},\bbr^N)}}
\end{eqnarray*}
under the condition~\eqref{eq0003},
where
\begin{eqnarray*}
D^{1,2}_{a,rad}(\bbr^N)=\{u\in D^{1,2}_{a}(\bbr^N)\mid u\text{ is radially symmetric}\}.
\end{eqnarray*}
Thus, $W(x)$ is always a solution of \eqref{eq0018} under the condition~\eqref{eq0003}.  It has been proved in \cite{CC1993,DEL2016} that $W(x)$ is the unique nonnegative solution of \eqref{eq0018} in $D^{1,2}_a(\bbr^N)$ either for $b_{FS}(a)\leq b<a+1$ with $a<0$ or for $a\leq b<a+1$ with $a\geq0$.  Moreover, it has also been proved in \cite{FS2003} that $W(x)$ is nondegenerate in $D^{1,2}_a(\bbr^N)$ under the condition~\eqref{eq0003}.  That is, up to scalar multiplications $CV(x)$,
\begin{eqnarray}\label{eq0010}
V(x):=\nabla W(x)\cdot x-(a_c-a)W(x)=\frac{\partial}{\partial\lambda}(\lambda^{-(a_c-a)}W(\lambda x))|_{\lambda=1}
\end{eqnarray}
is the only nonzero solution in $D^{1,2}_a(\bbr^N)$ to the linearization of \eqref{eq0018} around $W$ which is given by
\begin{eqnarray}\label{eq0017}
-div(|x|^{-a}\nabla u)=p|x|^{-b(p+1)}W^{p-1}u, \quad\text{in }\bbr^N.
\end{eqnarray}
By the transformation~\eqref{eq0007}, the linear equation~\eqref{eq0017} can be rewritten as follows:
\begin{eqnarray}\label{eq0016}
-\Delta_{\theta}v-\partial_t^2v+(a_c-a)^2v=p\Psi^{p-1}v, \quad\text{in }\mathcal{C},
\end{eqnarray}
where $\Delta_{\theta}$ is the Laplace-Beltrami operator on $\mathbb{S}^{N-1}$, $t=-\ln|x|$ and $\theta=\frac{x}{|x|}$ for $x\in\bbr^N\backslash\{0\}$, and
\begin{eqnarray}\label{eq0026}
\Psi(t)=\bigg(\frac{(p+1)(a_c-a)^2}{2}\bigg)^{\frac{1}{p-1}}\bigg(cosh(\frac{(a_c-a)(p-1)}{2}t)\bigg)^{-\frac{2}{p-1}}.
\end{eqnarray}
It follows from the transformation~\eqref{eq0007} that
\begin{eqnarray*}
\Psi_s'(t)=\Psi'(t-\log s)=\frac{\partial}{\partial t}\Psi(t-\log s)=-s\frac{\partial}{\partial s}\Psi(t-\log s)
\end{eqnarray*}
is the only nonzero solution of \eqref{eq0026} in $H^1(\mathcal{C})$ .

\section{Profile decompositions of nonnegative functions}
It is well-known that all minimizers of \eqref{eq0002} are positive in $\bbr^N\backslash\{0\}$.  Indeed, let
\begin{eqnarray*}
\mathcal{E}(u)=\frac{1}{2}\|u\|^2_{D^{1,2}_a(\bbr^N)}-\frac{1}{p+1}\|u\|^{p+1}_{L^{p+1}(|x|^{-b(p+1)},\bbr^N)},
\end{eqnarray*}
then by \eqref{eq0001}, $\mathcal{E}(u)$ is of class $C^2$ in $D^{1,2}_a(\bbr^N)$.  Since it is well-known that extremal functions of \eqref{eq0002} are ground states of \eqref{eq0018}, extremal functions of \eqref{eq0002} are also minimizers of the minimizing problem
\begin{eqnarray*}
c=\inf_{u\in\mathcal{N}}\mathcal{E}(u),
\end{eqnarray*}
where
\begin{eqnarray*}
\mathcal{N}=\{u\in D^{1,2}_a(\bbr^N)\backslash\{0\}\mid \mathcal{E}'(u)u=0\}
\end{eqnarray*}
is the usual Nehari manifold.  Since $p>1$ for $a\leq b<a+1$, it is standard to use the fibering maps to show the double-energy property of $\mathcal{E}(u)$, that is, $c_{sg}\geq2c$, where $c_{sg}=\inf_{u\in\mathcal{N}_{sg}}\mathcal{E}(u)$ with
\begin{eqnarray*}
\mathcal{N}_{sg}=\{u\in D^{1,2}_a(\bbr^N)\backslash\{0\}\mid \mathcal{E}'(u^\pm)u^\pm=0\}
\end{eqnarray*}
and $u^{\pm}=\max\{\pm u, 0\}$.
Thus, by the the double-energy property of $\mathcal{E}(u)$, all minimizers of $\mathcal{E}(u)$ in $\mathcal{N}$ at the energy level $c$ are nonnegative which implies that all extremal functions of \eqref{eq0002} are nonnegative.  It follows from the maximum principle that all extremal functions of \eqref{eq0002} are positive in $\bbr^N\backslash\{0\}$.

\vskip0.12in

As that of the Sobolev and Hardy-Sobolev inequalities, we have the following relatively compactness of minimizing sequences of \eqref{eq0002}.
\begin{proposition}\label{prop0001}
Suppose that $\{u_n\}\subset D^{1,2}_a(\bbr^N)$ be a minimizing sequence of \eqref{eq0002} either for $a<b<a+1$ with $a<0$ or for $a\leq b<a+1$ with $a\geq0$ and $a+b>0$.  Then there exists $\{\tau_n\}\subset(0, +\infty)$ such that $(u_n)_{\tau_n}\to u_0$ strongly in $D^{1,2}_a(\bbr^N)$ as $n\to\infty$ up to a subsequence, where $u_0$ is a minimizer of \eqref{eq0002}.  Moreover, we have $u_0=CW_{\tau_0}$ with some $\tau_0>0$ and $C\in\bbr\backslash\{0\}$ either for $b_{FS}(a)\leq b<a+1$ with $a<0$ or for $a\leq b<a+1$ with $a\geq0$ and $a+b>0$, where
$b_{FS}(a)$ is the Felli-Schneider curve given by \eqref{eqn991}.
\end{proposition}
\begin{proof}
Since the case $a\geq0$ is considered in \cite[Theorem~4]{WW2000}, we shall only give the proof for $a<0$.  Moreover, the proof is rather standard nowadays (cf. \cite{S2000}), so we only sketch it here.  Without loss of generality, we may assume that
\begin{eqnarray*}
\|u_n\|^2_{L^{p+1}(|x|^{-b(p+1)},\bbr^N)}=1.
\end{eqnarray*}
Then $\{u_n\}$ is bounded in $D^{1,2}_a(\bbr^N)$ and thus, $u_n\rightharpoonup \widehat{u}_0$ weakly in $D^{1,2}_a(\bbr^N)$ as $n\to\infty$ up to a subsequence.  Moreover, by the double-energy property of $\mathcal{E}(u)$, we may also assume that $\widehat{u}_0\geq0$ without loss of generality.  By \eqref{eq0007}, the corresponding $v_n\rightharpoonup v_0$ weakly in $H^1(\mathcal{C})$ as $n\to\infty$ up to a subsequence with $v_0\geq0$.  Since $a<b<a+1$ for $a<0$, we have $1<p<\frac{N+2}{N-2}$ by \eqref{eq0003}.  Thus, by \cite[Lemma~4.1]{CW2001}, there exists $\{\tau_n\}\subset\bbr$ such that
\begin{eqnarray*}
\overline{v}_n=v_n(t-\tau_n,\theta)\rightharpoonup \overline{v}_0\not=0\quad\text{weakly in $H^1(\mathcal{C})$ as $n\to\infty$}.
\end{eqnarray*}
It follows from the Brez\'is-Lieb lemma and the concavity of the function $t^{\frac{2}{p+1}}$ for $0<t<1$ with $p>1$ that
\begin{eqnarray*}
1+o_n(1)&=&C_{a,b,N}(\|\overline{v}_n-\overline{v}_0\|_{H^1(\mathcal{C})}^2+\|\overline{v}_0\|_{H^1(\mathcal{C})}^2)\\
&\geq&(1-\|\overline{v}_0\|_{L^{p+1}(\mathcal{C})}^{p+1}+o_n(1))^{\frac{2}{p+1}}+(\|\overline{v}_0\|_{L^{p+1}(\mathcal{C})}^{p+1})^{\frac{2}{p+1}}\\
&\geq&1+o_n(1),
\end{eqnarray*}
which implies that $\overline{v}_n\to\overline{v}_0$ strongly in $L^{p+1}(\mathcal{C})$ as $n\to\infty$.  Correspondingly, by \eqref{eq0007}, we have $(u_n)_{\tau_n}\to u_0$ strongly in $L^{p+1}(|x|^{-b(p+1)},\bbr^N)$ as $n\to\infty$.  It is then easy to show that $u_0$ is a minimizer of \eqref{eq0002}.  In the cases $b_{FS}(a)\leq b<a+1$ with $a<0$ or $a\leq b<a+1$ with $a\geq0$ and $a+b>0$, $W$ is the unique minimizer of \eqref{eq0002} up to dilations $u_{\tau}(x)=\tau^{a_c-a}u(\tau x)$ and scalar multiplications $Cu(x)$.  Thus, we must have
$u_0=C W_{\tau_0}$ with some $\tau_0>0$ and $C\in\bbr\backslash\{0\}$.
\end{proof}

\vskip0.12in

As the well-known results of profile decompositions to the Sobolev inequality due to Struwe (cf. \cite{S1984,S2000}), we have the following profile decompositions of \eqref{eq0018} for nonnegative functions which, to our best knowledge, is new.
\begin{proposition}\label{prop0002}
Let $\{w_n\}$ be a nonnegative $(PS)$ sequence of $\mathcal{E}(u)$ with
\begin{eqnarray*}
(\nu-\frac12)(C_{a,b,N}^{-1})^{\frac{p+1}{p-1}}<\|w_n\|_{D^{1,2}_a(\bbr^N)}^2<(\nu+\frac12)(C_{a,b,N}^{-1})^{\frac{p+1}{p-1}}
\end{eqnarray*}
with some $\nu\in\bbn$ either for $b_{FS}(a)\leq b<a+1$ with $a<0$ or for $a\leq b<a+1$ with $a\geq0$ and $a+b>0$ where
$b_{FS}(a)$ is the Felli-Schneider curve given by \eqref{eqn991}.
Then there exists $\{\tau_{i,n}\}\subset\bbr_+$, satisfying
\begin{eqnarray*}
\min_{i\not=j}\bigg\{\max\bigg\{\frac{\tau_{i,n}}{\tau_{j,n}},\frac{\tau_{j,n}}{\tau_{i,n}}\bigg\}\bigg\}\to+\infty
\end{eqnarray*}
as $n\to\infty$ for $\nu\geq2$, such that
\begin{enumerate}
\item[$(1)$]\quad $w_n=\sum_{i=1}^{\nu}(W)_{\tau_{i,n}}+o_n(1)$ in $D^{1,2}_a(\bbr^N)$.
\item[$(2)$]\quad $\|w_n\|_{D^{1,2}_a(\bbr^N)}^2=\nu\|W\|_{D^{1,2}_a(\bbr^N)}^2+o_n(1)$.
\end{enumerate}
\end{proposition}
\begin{proof}
Since $p=\frac{N+2}{N-2}$ for $a=b$ and $p<\frac{N+2}{N-2}$ for $a<b$, We shall divide the proof into two parts which is devoted to the case $p<\frac{N+2}{N-2}$ and $p=\frac{N+2}{N-2}$, respectively.

\vskip0.12in

{\it The case $p<\frac{N+2}{N-2} (a<b)$.}

In this case, we use the transformation~\eqref{eq0007} to $w_n$.  Then the related $\widetilde{w}_n(t,\theta)$ satisfy
\begin{eqnarray*}
(\nu-\frac12)(C_{a,b,N}^{-1})^{\frac{p+1}{p-1}}<\|\widetilde{w}_n\|_{H^1(\mathcal{C})}^2<(\nu+\frac12)(C_{a,b,N}^{-1})^{\frac{p+1}{p-1}}
\end{eqnarray*}
and $\mathcal{J}'(\widetilde{w}_n)\to0$ in $H^{-1}(\mathcal{C})$ as $n\to\infty$, where $H^{-1}(\mathcal{C})$ is the dual space of $H^{1}(\mathcal{C})$ and
\begin{eqnarray*}
\mathcal{J}(u)=\frac{1}{2}\|u\|^2_{H^1(\mathcal{C})}-\frac{1}{p+1}\|u\|^{p+1}_{L^{p+1}(\mathcal{C})}.
\end{eqnarray*}
Since $p<\frac{N+2}{N-2}$ and $\Psi(t)$ is the unique nonnegative solution of \eqref{eq0026} in $H^1(\mathcal{C})$ either for $b_{FS}(a)\leq b<a+1$ with $a<0$ or for $a<b<a+1$ with $a\geq0$ and $a+b>0$, the conclusion then follows from \eqref{eq0007} and adapting \cite[Lemma~4.1]{CW2001} in a standard way.

\vskip0.12in

{\it The case $p=\frac{N+2}{N-2} (a=b)$.}

In this case, we have $a>0$ by the assumptions.  Moreover, \cite[Lemma~4.1]{CW2001} is invalid to drive the conclusion and thus, we shall mainly follows Struwe's idea in proving \cite[Theorem~3.1]{S2000}.  However, according to the singular potential $|x|^{-2a}$, the argument is more involved. Let $U_\ve$ be the standard Aubin-Talanti bubbles, that is,
\begin{eqnarray*}
U_\ve=[N(N-2)]^{\frac{N-2}{4}}\bigg(\frac{\ve}{\ve^2+|x|^2}\bigg)^{\frac{N-2}{2}}.
\end{eqnarray*}
By \cite[Lemma~1]{WW2000},
\begin{eqnarray}\label{eqn992}
C_{a,a,N}^{-1}<S\quad\text{for }a>0.
\end{eqnarray}
Thus, there exists $R_\ve>0$ such that
\begin{eqnarray}\label{eq0054}
\int_{B_{R_\ve}(0)}|\nabla U_\ve|^2dx>\frac{1}{L_{R_\ve}}(C_{a,b,N}^{-1})^{\frac{N}{2}},
\end{eqnarray}
where $L_{R_\ve}$ is the number such that the ball $B_{2R_\ve}(0)$ is covered by $L_{R_\ve}$ balls with radius $R_\ve$.  Here, we have used the fact that $2\leq L_{R_\ve}\leq 2^N$.  Let
\begin{eqnarray*}
Q_n(r)=\sup_{y\in\bbr^N}\int_{B_r(y)}|x|^{-2a}|\nabla w_n|^2dx
\end{eqnarray*}
be the well-known concentration function of $w_n$.  Since
\begin{eqnarray*}
(\nu-\frac12)(C_{a,a,N}^{-1})^{\frac{N}{2}}<\|w_n\|_{D^{1,2}_a(\bbr^N)}^2<(\nu+\frac12)(C_{a,a,N}^{-1})^{\frac{N}{2}}
\end{eqnarray*}
for some $\nu\in\bbn$, we can choose $r_n>0$ and $y_n\in\bbr^N$ such that
\begin{eqnarray*}
Q_n(r_n)=\int_{B_{r_n}(y_n)}|x|^{-2a}|\nabla w_n|^2dx=\frac{1}{2L_{R_\ve}}(C_{a,a,N}^{-1})^{\frac{N}{2}}.
\end{eqnarray*}
Let
\begin{eqnarray*}
v_n=(r_nR_\ve^{-1})^{-(a_c-a)}w_n(r_nR_\ve^{-1}x),
\end{eqnarray*}
then
\begin{eqnarray}\label{eq0050}
\sup_{y\in\bbr^N}\int_{B_{R_\ve}(y)}|x|^{-2a}|\nabla v_n|^2dx=\int_{B_{R_\ve}(\frac{R_\ve y_n}{r_n})}|x|^{-2a}|\nabla v_n|^2dx=\frac{1}{2L_{R_\ve}}(C_{a,a,N}^{-1})^{\frac{N}{2}}.
\end{eqnarray}
Since $\|\cdot\|_{D^{1,2}_a}$ is invariant under the dilation $u_{\tau}(x)=\tau^{a_c-a}u(\tau x)$, $\{v_n\}$ is bounded in $D^{1,2}(\bbr^N)$ and thus, $v_n\rightharpoonup v_0$ weakly in $D^{1,2}(\bbr^N)$ as $n\to\infty$ up to a subsequence.  Clearly, $v_0\geq0$.  We define $\varpi_n=(v_n-v_0)\varphi$, where $\varphi$ is a smooth cut-off function such that $\varphi=1$ in $B_{R_\ve}(z)$ and $\varphi=0$ in $B^c_{\frac{3}{2}R_\ve}(z)$ for any $z\in\bbr^N$.  Then by \cite[Lemma~2]{WW2000},
\begin{eqnarray*}
\|\varpi_n\|_{D^{1,2}_a(\bbr^N)}^2\lesssim\|v_n-v_0\|_{D^{1,2}_a(\bbr^N)}^2+\int_{B_{2R_\ve}(z)\backslash B_{R_\ve}(z)}|x|^{-2a}|v_n-v_0|^2dx\lesssim1
\end{eqnarray*}
and
\begin{eqnarray*}
\int_{B_{2R_\ve}(z)\backslash B_{R_\ve}(z)}|x|^{-2a}|v_n-v_0|^2dx\to0\quad\text{as }n\to\infty.
\end{eqnarray*}
Note that by the fact that $\mathcal{E}'(w_n)\to0$ in $D^{-1,2}_a(\bbr^N)$ as $n\to\infty$ and the invariance of $\mathcal{E}(w_n)$ under the the dilation $u_{\tau}(x)=\tau^{a_c-a}u(\tau x)$, $\mathcal{E}'(v_n)\to0$ in $D^{-1,2}_a(\bbr^N)$ as $n\to\infty$.  It follows that $\mathcal{E}'(v_0)=0$, which, together with the Brez\'is-Lieb lemma, implies
\begin{eqnarray}
o_n(1)&=&\int_{\bbr^N}(|x|^{-2a}\nabla(v_n-v_0)\nabla (\varpi_n\varphi)-|x|^{-\frac{2Na}{N-2}}(v_n^{\frac{N+2}{N-2}}-v_0^{\frac{N+2}{N-2}})\varpi_n\varphi)dx\notag\\
&=&\int_{\bbr^N}(|x|^{-2a}|\nabla \varpi_n|^2-|x|^{-\frac{2Na}{N-2}}|v_n-v_0|^{\frac{4}{N-2}}|\varpi_n|^{2})dx+o_n(1)\notag\\
&\geq&\|\varpi_n\|_{D^{1,2}_a(\bbr^N)}^2-\|\varpi_{n,*}\|_{L^{\frac{2N}{N-2}}(|x|^{-\frac{2Na}{N-2}},\bbr^N)}^{\frac{4}{N-2}}\|\varpi_{n}\|_{L^{\frac{2N}{N-2}}(|x|^{-\frac{2Na}{N-2}},\bbr^N)}^2\notag\\
&&+o_n(1).\label{eq0051}
\end{eqnarray}
Here, $\varpi_{n,*}=(v_n-v_0)\varphi_*$, where $\varphi_*$ is a smooth cut-off function such that $\varphi_*=1$ in $B_{\frac{3}{2}R_\ve}(z)$ and $\varphi_*=0$ in $B^c_{2R_\ve}(z)$.  By the Brez\'is-Lieb lemma once more,
\begin{eqnarray*}
\|\varpi_{n,*}\|_{L^{\frac{2N}{N-2}}(|x|^{-\frac{2Na}{N-2}},\bbr^N)}^{\frac{2N}{N-2}}&\leq&\int_{B_{2R_\ve}(z)}|x|^{-\frac{2Na}{N-2}}|v_n-v_0|^{\frac{2N}{N-2}}dx\\
&\leq&\int_{B_{2R_\ve}(z)}|x|^{-\frac{2Na}{N-2}}|v_n|^{\frac{2N}{N-2}}dx+o_n(1).
\end{eqnarray*}
It follows from the CKN inequality~\eqref{eq0001}, \eqref{eq0050} and \eqref{eq0051} that $\varpi_n\to0$ strongly in $D^{1,2}_a(\bbr^N)$ as $n\to\infty$, which implies that $v_n\to v_0$ strongly in $D^{1,2}_a(B_{R_0}(z))$ as $n\to\infty$ for any $z\in\bbr^N$.  Thus, by a standard covering argument, $v_n\to v_0$ strongly $D^{1,2}_{a,loc}(\bbr^N)$ as $n\to\infty$.  We claim that $v_0\not=0$.  Assume the contrary that $v_0=0$, then by \eqref{eq0050}, $|\frac{R_\ve y_n}{r_n}|\to+\infty$ as $n\to\infty$.  It follows that
\begin{eqnarray*}
\int_{B_{\frac12|\frac{R_\ve y_n}{r_n}|}(\frac{R_\ve y_n}{r_n})}|\nabla(\frac{v_n}{|\frac{R_\ve y_n}{r_n}|^a})|^2\sim\int_{B_{\frac12|\frac{R_\ve y_n}{r_n}|}(\frac{R_\ve y_n}{r_n})}|x|^{-2a}|\nabla v_n|^2\lesssim\|v_n\|_{D^{1,2}_a(\bbr^N)}^2,
\end{eqnarray*}
which implies that $\{\overline{v}_n\}$ is bounded in $D^{1,2}_{loc}(\bbr^N)$ and thus, $\overline{v}_n\rightharpoonup \overline{v}_0$ weakly in $D^{1,2}_{loc}(\bbr^N)$ as $n\to\infty$, where $\overline{v}_n=\frac{v_n(x+\frac{R_\ve y_n}{r_n})}{|\frac{R_\ve y_n}{r_n}|^a}$.  Now, by $\mathcal{E}'(v_n)\to0$ in $D^{-1,2}_a(\bbr^N)$ as $n\to\infty$, we know that
\begin{eqnarray}\label{eq0053}
-\Delta\overline{v}_n=\overline{v}_n^{\frac{N+2}{N-2}}+o_n(1)\quad\text{in }D^{1,2}_{loc}(\bbr^N).
\end{eqnarray}
By \eqref{eqn992}, \eqref{eq0050} and $|\frac{R_\ve y_n}{r_n}|\to+\infty$ as $n\to\infty$,
\begin{eqnarray}
\int_{B_{R_\ve}(0)}|\nabla \overline{v}_n|^2dx&=&\sup_{y\in\bbr^N}\int_{B_{R_\ve}(y)}|\nabla \overline{v}_n|^2dx\notag\\
&=&\frac{1}{2L_{R_\ve}}(C_{a,b,N}^{-1})^{\frac{N}{2}}+o_n(1)\label{eq0052}\\
&<&\frac{1}{2L_{R_\ve}}S^{\frac{N}{2}}+o_n(1)\notag
\end{eqnarray}
Thus, by applying similar arguments as that used for \eqref{eq0051} to \eqref{eq0053}, we can show that $\overline{v}_n\to \overline{v}_0$ strongly in $D^{1,2}(B_{R_\ve}(0))$ as $n\to\infty$.  By \eqref{eq0052}, $\overline{v}_0\not=0$ and thus, $\overline{v}_0=U_\ve$ for some $\ve>0$ by \eqref{eq0053}.  It is impossible since \eqref{eq0054} and \eqref{eq0052} hold at the same time now.  Therefore, we must have $v_0\not=0$.  Since $v_0\geq0$, by \cite[Theorem~B and Proposition~4.4]{CC1993} and \cite[Theorem~1.2]{DEL2016}, we have $v_0=W$ either for $b_{FS}(a)\leq b<a+1$ with $a<0$ or for $a\leq b<a+1$ with $a\geq0$ and $a+b>0$.  Thus, $w_n\rightharpoonup(W)_{r_nR_\ve^{-1}}$ weakly in $D^{1,2}_a(\bbr^N)$ as $n\to\infty$.  By running the above argument to $w_n-(W)_{r_nR_\ve^{-1}}$, we will arrive at that $w_n\rightharpoonup(W)_{r_{n,1}r_nR_\ve^{-1} R_{\ve_1}^{-1}}+(W)_{r_{n,1}R_{\ve_1}^{-1}}$ weakly in $D^{1,2}_a(\bbr^N)$ as $n\to\infty$ for some $r_{n,1}>0$ and $\ve_1>0$.  The conclusion then follows from iterating the above arguments for $\nu$ times and using the fact that $W(|x|)$ is the unique nonnegative solution of \eqref{eq0018} in $D^{1,2}_a(\bbr^N)$ either for $b_{FS}(a)\leq b<a+1$ with $a<0$ or for $a\leq b<a+1$ with $a\geq0$ and $a+b>0$.
\end{proof}

\section{Stability of CKN inequality in the functional inequality setting}
It is well-known that the minimizing problem~\eqref{eq0002} and the equation~\eqref{eq0018} are invariant under the dilation $u_\tau(x)=\tau^{\frac{N-2-2a}{2}}u(\tau x)$.  Thus, the smooth manifold
\begin{eqnarray*}
\mathcal{Z}=\{cW_\tau(x)\mid c\in\bbr\backslash\{0\}\text{ and }\tau>0\}
\end{eqnarray*}
is all extremal functions of the minimizing problem~\eqref{eq0002}.  Let
\begin{eqnarray*}
d^2(u)=\inf_{c\in\bbr,\ \ \tau>0}\|u-cW_\tau\|^2_{D^{1,2}_a(\bbr^N)},
\end{eqnarray*}
where $u\in D^{1,2}_a(\bbr^N)$.  Then we have the following stability for the CKN inequality~\eqref{eq0001}.
\begin{proposition}\label{prop0003}
Let $e(u):=\|u\|^2_{D^{1,2}_a(\bbr^N)}-C_{a,b,N}^{-1}\|u\|^2_{L^{p+1}(|x|^{-b(p+1)},\bbr^N)}$.  Then $e(u)\gtrsim d^2(u)$ for all $u\in D^{1,2}_a(\bbr^N)$ in the following two cases:
\begin{enumerate}
\item[$(1)$]\quad $b_{FS}(a)\leq b<a+1$ with $a<0$,
\item[$(2)$]\quad $a\leq b<a+1$ with $a\geq0$ and $a+b>0$.
\end{enumerate}
\end{proposition}
\begin{proof}
The proof mainly follows the arguments in \cite{BE1991} for the stability of the Sobolev inequality.  It is easy to see that $d^2(u)$ can be attained by some $c_0\not=0$ and $\tau_0>0$.  Indeed,
\begin{eqnarray*}
\|u-cW_\tau\|^2_{D^{1,2}_a(\bbr^N)}=\|u\|^2_{D^{1,2}_a(\bbr^N)}+c^2\|W_\tau\|^2_{D^{1,2}_a(\bbr^N)}-c\langle u,W_\tau\rangle_{D^{1,2}_a(\bbr^N)}.
\end{eqnarray*}
Thus, by taking $(c,\tau)\in(\bbr, \bbr^+)$ such that $c\langle u,W_\tau\rangle_{D^{1,2}_a(\bbr^N)}>0$ with $|c|>0$ sufficiently small, we have $d^2(u)<\|u\|^2_{D^{1,2}_a(\bbr^N)}$.
By the invariance of the norm $\|\cdot\|_{D^{1,2}_a(\bbr^N)}$ under the dilation $u_\tau(x)=\tau^{a_c-a}u(\tau x)$,
\begin{eqnarray}
\|u-cW_\tau\|^2_{D^{1,2}_a(\bbr^N)}&=&\|u\|^2_{D^{1,2}_a(\bbr^N)}+c^2\|W_\tau\|^2_{D^{1,2}_a(\bbr^N)}-c\langle u,W_\tau\rangle_{D^{1,2}_a(\bbr^N)}\label{eq0056}\\
&\geq&\|u\|^2_{D^{1,2}_a(\bbr^N)}+c^2\|W\|^2_{D^{1,2}_a(\bbr^N)}-|c|\|u\|_{D^{1,2}_a(\bbr^N)}\|W\|_{D^{1,2}_a(\bbr^N)}.\notag
\end{eqnarray}
Thus, the minimizing sequence of $d^2(u)$, say $\{(c_n,\tau_n)\}$, must satisfy $|c_n|\sim1$.  On the other hand,
\begin{eqnarray*}
|\int_{|\tau x|\leq \rho}|x|^{-2a}\nabla u\nabla W_\tau|&\leq&\int_{|y|\leq\rho}|y|^{-2a}|\nabla u_{\frac{1}{\tau}}(y)\nabla W(y)|\notag\\
&\leq&\|u\|_{D^{1,2}_a(\bbr^N)}\bigg(\int_{|y|\leq\rho}|y|^{-2a}|\nabla W|^2\bigg)^{\frac12}\notag\\
&=&o_\rho(1)
\end{eqnarray*}
as $\rho\to0$ which is uniformly for $\tau>0$ and
\begin{eqnarray*}
|\int_{|\tau x|\geq \rho}|x|^{-2a}\nabla u\nabla W_\tau|
&\leq&\|W\|_{D^{1,2}_a(\bbr^N)}\bigg(\int_{|x|\geq \frac{\rho}{\tau}}|x|^{-2a}|\nabla u|^2\bigg)^{\frac12}\notag\\
&=&o_\tau(1)
\end{eqnarray*}
as $\tau\to0$ for any fixed $\rho>0$.  By taking $\tau\to0$ first and $\rho\to0$ next, we have $|\int_{\bbr^N}|x|^{-2a}\nabla u\nabla W_\tau|\to0$ as $\tau\to0$.
Note that $1+|\tau x|\sim 1$ for $|\tau x|\lesssim1$, by \eqref{eq0004},
\begin{eqnarray*}
|\int_{|\tau x|\leq R}|x|^{-2a}\nabla u\nabla W_\tau|\lesssim\tau^{a_c-a}\|u\|_{D^{1,2}_a(\bbr^N)}\bigg(\int_0^{\frac{R}{\tau}}r^{-2a+N-1}\bigg)^{\frac{1}{2}}=o_\tau(1)
\end{eqnarray*}
as $\tau\to+\infty$ for any fixed $R>0$ and
\begin{eqnarray*}
|\int_{|\tau x|\geq R}|x|^{-2a}\nabla u\nabla W_\tau|&\leq&\int_{|y|\geq R}|y|^{-2a}|\nabla u_{\frac{1}{\tau}}(y)\nabla W(y)|\notag\\
&\leq&\|u\|_{D^{1,2}_a(\bbr^N)}\bigg(\int_{|y|\geq R}|y|^{-2a}|\nabla W|^2\bigg)^{\frac12}\notag\\
&=&o_R(1)
\end{eqnarray*}
as $R\to+\infty$  which is uniformly for $\tau>0$.  Thus, by taking $\tau\to+\infty$ first and $R\to+\infty$ next, we also have $|\int_{\bbr^N}|x|^{-2a}\nabla u\nabla W_\tau|\to0$ as $\tau\to+\infty$.  It follows from \eqref{eq0056} and $d^2(u)<\|u\|^2_{D^{1,2}_a(\bbr^N)}$ that the minimizing sequence $\{(c_n,\tau_n)\}$ must satisfy $|\tau_n|\sim1$.  Thus, $d^2(u)$ can be attained by some $c_0\not=0$ and $\tau_0>0$, which implies
\begin{eqnarray*}
\langle u, c_0W_{\tau_0}\rangle_{D^{1,2}_a(\bbr^N)}=\|c_0W_{\tau_0}\|_{D^{1,2}_a(\bbr^N)}^2\quad\text{and}\quad\langle u, \partial_{\tau}W_{\tau}|_{\tau=\tau_0}\rangle_{D^{1,2}_a(\bbr^N)}=0.
\end{eqnarray*}
Note that
\begin{eqnarray*}
\mathcal{T}_{W_{\tau_0}}\mathcal{Z}=\text{span}\{\partial_{\tau}W_{\tau}|_{\tau=\tau_0}\},
\end{eqnarray*}
and $\partial_{\tau}W_{\tau}|_{\tau=\tau_0}=V_{\tau_0}$, where $V(x)$ is given by \eqref{eq0010}.  Thus, by the nondegneracy of $W_{\tau_0}$ in $D^{1,2}_a(\bbr^N)$,
\begin{eqnarray}\label{eq0057}
u=c_0W_{\tau_0}+\phi_{\tau_0}
\end{eqnarray}
in $D^{1,2}_a(\bbr^N)$, where
\begin{eqnarray}\label{eq0058}
\langle \phi_{\tau_0}, W_{\tau_0}\rangle_{D^{1,2}_a(\bbr^N)}=\langle \phi_{\tau_0}, V_{\tau_0}\rangle_{D^{1,2}_a(\bbr^N)}=0.
\end{eqnarray}
It follows that $d^2(u)=\|\phi_{\tau_0}\|_{D^{1,2}_a(\bbr^N)}^2$.  Since $W_{\tau_0}$ is the ground state, the Morse index of $W_{\tau_0}$ is equal to $1$.  It follows from the nondegneracy of $W_{\tau_0}$ in $D^{1,2}_a(\bbr^N)$ that
\begin{eqnarray}\label{eq0011}
\|\phi_{\tau_0}\|_{D^{1,2}_a(\bbr^N)}^2>p\int_{\bbr^N}|x|^{-b(p+1)}W_{\tau_0}^{p-1}\phi_{\tau_0}^2.
\end{eqnarray}
Let us first consider the case that $d(u)>0$ is sufficiently small, then by the elementary inequality
\begin{eqnarray*}
\bigg||\alpha+\beta|^q-|\alpha|^q-q|\alpha|^{q-2}\alpha\beta-\frac{q(q-1)}{2}|\alpha|^{q-2}\beta^2\bigg|\lesssim|\beta|^{q}+|\alpha|^{q-3}|\beta|^3\chi_{q\geq3}
\end{eqnarray*}
for $q>2$, where $\chi_{q\geq3}=1$ for $q\geq3$ and $\chi_{q\geq3}=0$ for $2<q<3$, and the CKN inequality~\eqref{eq0001},
\begin{eqnarray*}
\|W_{\tau_0}+\frac{\phi_{\tau_0}}{c_0}\|^{p+1}_{L^{p+1}(|x|^{-b(p+1)},\bbr^N)}&=&\|W_{\tau_0}\|^{p+1}_{L^{p+1}(|x|^{-b(p+1)},\bbr^N)}+o(d^{2}(u))\\
&&+(p+1)\int_{\bbr^N}|x|^{-b(p+1)}W_{\tau_0}^{p}\frac{\phi_{\tau_0}}{c_0}\notag\\
&&+\frac{p(p+1)}{2}\int_{\bbr^N}|x|^{-b(p+1)}W_{\tau_0}^{p-1}(\frac{\phi_{\tau_0}}{c_0})^2,
\end{eqnarray*}
which, together with $\langle \phi_{\tau_0}, W_{\tau_0}\rangle_{D^{1,2}_a(\bbr^N)}=0$ and the fact that $W_{\tau_0}$ is a solution of \eqref{eq0018}, implies that
\begin{eqnarray*}
\|W_{\tau_0}+\frac{\phi_{\tau_0}}{c_0}\|^{p+1}_{L^{p+1}(|x|^{-b(p+1)},\bbr^N)}&=&\|W_{\tau_0}\|^{p+1}_{L^{p+1}(|x|^{-b(p+1)},\bbr^N)}+o(d^2(u))\\
&&+\frac{p(p+1)}{2}\int_{\bbr^N}|x|^{-b(p+1)}W_{\tau_0}^{p-1}(\frac{\phi_{\tau_0}}{c_0})^2.
\end{eqnarray*}
On the other hand, by the fact that $W_{\tau_0}$ is a solution of \eqref{eq0018} and it is also a minimizer of \eqref{eq0002}, we have
\begin{eqnarray*}
C_{a,b,N}^{-1}=\frac{\|W_{\tau_0}\|^2_{D^{1,2}_a(\bbr^N)}}{\|W_{\tau_0}\|^2_{L^{p+1}(|x|^{-b(p+1)},\bbr^N)}}=\|W_{\tau_0}\|^{p-1}_{L^{p+1}(|x|^{-b(p+1)},\bbr^N)}.
\end{eqnarray*}
It follows from \eqref{eq0057}, \eqref{eq0058} and \eqref{eq0011} that for $d(u)>0$ is sufficiently small,
\begin{eqnarray*}\label{eq0012}
e(u)&:=&\|u\|^2_{D^{1,2}_a(\bbr^N)}-C_{a,b,N}^{-1}\|u\|^2_{L^{p+1}(|x|^{-b(p+1)},\bbr^N)}\notag\\
&=&c_0^2\bigg(\|\frac{\phi_{\tau_0}}{c_0}\|^2_{D^{1,2}_a(\bbr^N)}-p\int_{\bbr^N}|x|^{-b(p+1)}W_{\tau_0}^{p-1}(\frac{\phi_{\tau_0}}{c_0})^2+o(d^2(u))\bigg)\notag\\
&\gtrsim&d^2(u).
\end{eqnarray*}
It remains to consider the case $d(u)\gtrsim1$.  Assume that $e(u)\gtrsim d^2(u)$ does not hold for all $u\in D^{1,2}_a(\bbr^N)$.  Then there exists $\{u_n\}\subset D^{1,2}_a(\bbr^N)\backslash\{0\}$ such that $e(u_n)=o(d^2(u_n))$.  Thus, $e(u_n)\to0$ as $n\to\infty$ in this case.  It follows that $\{u_n\}$ is a minimizing sequence of \eqref{eq0002}.  By Proposition~\ref{prop0001}, we have $d(u_n)\to0$ as $n\to\infty$, which is a contradiction.
\end{proof}

\vskip0.2in

We close this section by the proof of Theorem~\ref{thm0001}.

\vskip0.12in

\noindent\textbf{Proof of Theorem~\ref{thm0001}:} It follows immediately from Proposition~\ref{prop0003}.
\hfill$\Box$

\section{Stability of profile decompositions to nonnegative functions}
\subsection{The one-bubble case}
In this section, we will consider the one-bubble case and prove the following result.
\begin{proposition}\label{propn0001}
Let $v\in H^1(\mathcal{C})$ be nonnegative such that
\begin{eqnarray*}
\frac12(C_{a,b,N}^{-1})^{\frac{p+1}{p-1}}<\|v\|_{H^1(\mathcal{C})}^2<\frac32(C_{a,b,N}^{-1})^{\frac{p+1}{p-1}}\quad\text{and}\quad \|f\|_{H^{-1}(\mathcal{C})}\leq\delta
\end{eqnarray*}
for $\delta>0$ sufficiently small, where $f=-\Delta_{\theta}v-\partial_t^2v+(a_c-a)^2v-v^{p}$.  Then either for $b_{FS}(a)\leq b<a+1$ with $a<0$ or for $a\leq b<a+1$ with $a\geq0$ and $a+b>0$, we have
\begin{eqnarray}\label{eqn20001}
d_0(v)\lesssim\|f\|_{H^{-1}(\mathcal{C})}
\end{eqnarray}
where $d_0^2(v)=\inf_{s\in\bbr}\|v-\Psi_s\|^2_{H^1(\mathcal{C})}$.
\end{proposition}
\begin{proof}
We shall mainly adapt the ideas in \cite{CFM2017} to prove this proposition.  As that in the proof of Proposition~\ref{prop0003},
\begin{eqnarray*}
\widetilde{d}_0(v)=\inf_{c\in\bbr,\ \ s\in\bbr}\|v-c\Psi_s\|_{H^1(\mathcal{C})}
\end{eqnarray*}
is attained by some $c_0\not=0$ and $s_0\in\bbr$, which implies that $v=c_0\Psi_{s_0}+\psi_0$ and $\widetilde{d}_0(v)=\|\psi_0\|_{H^1(\mathcal{C})}$, where
\begin{eqnarray}\label{eqn20000}
\langle \Psi_{s_0}, \psi_0\rangle_{H^1(\mathcal{C})}=0\quad\text{and}\quad \langle \Psi_{s_0}', \psi_0\rangle_{H^1(\mathcal{C})}=0.
\end{eqnarray}
By Proposition~\ref{prop0002}, we also have that $\|\psi_0\|_{H^1(\mathcal{C})}\to0$ and $c_0=1+\alpha_0$ with $\alpha_0\to0$ as $\delta\to0$.  Now, by the orthogonal condition~\eqref{eqn20000},
\begin{eqnarray*}
\|\psi_0\|_{H^1(\mathcal{C})}^2=\langle\psi_0, v\rangle_{H^1(\mathcal{C})}=\int_{\mathcal{C}}v^{p}\psi_0+\int_{\mathcal{C}}f\psi_0.
\end{eqnarray*}
By the Taylor expansion and some elementary inequalities,
\begin{eqnarray*}
\int_{\mathcal{C}}v^{p}\psi_0=c_0^{p}\int_{\mathcal{C}}\Psi_{s_0}^{p}\psi_0+pc_0^{p-1}\int_{\mathcal{C}}\Psi_{s_0}^{p-1}\psi_0^2+O(\|\psi_0\|_{H^1(\mathcal{C})}^{\sigma_p+1}),
\end{eqnarray*}
where $\sigma_p=2$ for $p\geq2$ and $\sigma_p=p$ for $1<p<2$.  Since $\Psi$ is a solution of \eqref{eq0006},
\begin{eqnarray*}
\int_{\mathcal{C}}v^{p}\psi_0=pc_0^{p-1}\int_{\mathcal{C}}\Psi_{s_0}^{p-1}\psi_0^2+O(\|\psi_0\|_{H^1(\mathcal{C})}^{\sigma_p+1}).
\end{eqnarray*}
Note that $\Psi$ is the ground state of \eqref{eq0006}, thus, the Morse index of $\Psi$ is equal to $1$.  It follows from the orthogonal condition~\eqref{eqn20000} and the nondegeneracy of $\Psi$ in $H^1(\mathcal{C})$ that
\begin{eqnarray*}
p\int_{\mathcal{C}}\Psi_{s_0}^{p-1}\psi_0^2<\|\psi_0\|_{H^1(\mathcal{C})}^2.
\end{eqnarray*}
Thus, by $\|\psi_0\|_{H^1(\mathcal{C})}\to0$ and $c_0=1+\alpha_0$ with $\alpha_0\to0$ as $\delta\to0$,
\begin{eqnarray}\label{eqn5555}
\|\psi_0\|_{H^1(\mathcal{C})}\lesssim\|f\|_{H^{-1}(\mathcal{C})}
\end{eqnarray}
for $\delta>0$ sufficiently small.
On the other hand, we have
\begin{eqnarray*}
\|v\|_{H^1(\mathcal{C})}^2=\|v\|_{L^{p+1}(\mathcal{C})}^{p+1}+\int_{\mathcal{C}}fv.
\end{eqnarray*}
Since $c_0=1+\alpha_0$ with $\alpha_0\to0$ as $\delta\to0$, by the orthogonal condition~\eqref{eqn20000},
\begin{eqnarray*}
\|v\|_{H^1(\mathcal{C})}^2=(1+2\alpha_0+O(\alpha_0^2))\|\Psi\|_{H^1(\mathcal{C})}^2+\|\psi_0\|_{H^1(\mathcal{C})}^2.
\end{eqnarray*}
By the Taylor expansion, the orthogonal condition~\eqref{eqn20000} and some elementary inequalities,
\begin{eqnarray*}
\|v\|_{L^{p+1}(\mathcal{C})}^{p+1}=(1+(p+1)\alpha_0+O(\alpha_0^2))\|\Psi\|_{L^{p+1}(\mathcal{C})}^{p+1}+O(\|\psi_0\|_{H^1(\mathcal{C})}^{2}),
\end{eqnarray*}
which, together with \eqref{eqn5555} and the fact that $\Psi$ is a solution of \eqref{eq0006}, implies that
\begin{eqnarray*}
(p-1)|\alpha_0|\lesssim\|f\|_{H^{-1}(\mathcal{C})}^2+\|f\|_{H^{-1}(\mathcal{C})}\sim\|f\|_{H^{-1}(\mathcal{C})}
\end{eqnarray*}
for $\delta>0$ sufficiently small.
Now, \eqref{eqn20001} then follows from rewriting $v=\Psi_{s_0}+\alpha_0\Psi_{s_0}+\psi_0$.
\end{proof}

\subsection{The multi-bubble case}
Let us first compute the interaction of two bubbles, which plays an important role in the stability for the multi-bubble case.
\begin{lemma}\label{lem0001}
Let $W_{\tau_1}$ and $W_{\tau_2}$ be two bubbles such that $\tau_1\not=\tau_2$.  Then
\begin{eqnarray*}
\langle W_{\tau_1}, W_{\tau_2}\rangle_{D^{1,2}_a(\bbr^N)}\sim\bigg(\frac{\min\{\tau_1,\tau_2\}}{\max\{\tau_1,\tau_2\}}\bigg)^{a_c-a}.
\end{eqnarray*}
Moreover, by the transformation~\eqref{eq0007}, we also have
\begin{eqnarray}\label{eq0025}
\langle \Psi_{s_1}, \Psi_{s_2}\rangle_{H^{1}(\mathcal{C})}\sim e^{-(a_c-a)|s_1-s_2|},
\end{eqnarray}
where $\Psi(t)$ is given by \eqref{eq0026}, $\Psi_s(t)=\Psi(t-s)$ and $s_i=\ln\tau_i$.
\end{lemma}
\begin{proof}
Without loss of generality, we may assume that $\tau_1=1$ and $\tau_2:=\tau<1$ by the invariance of $\langle \cdot, \cdot\rangle_{D^{1,2}_a(\bbr^N)}$ under the dilation $u_\tau=\tau^{a_c-a}u(\tau x)$.  Since $W$ is a solution of \eqref{eq0018},
\begin{eqnarray*}
\langle W, W_{\tau}\rangle_{D^{1,2}_a(\bbr^N)}&=&\int_{\bbr^N}|x|^{-b(p+1)}W^{p}W_\tau\\
&=&\int_{|x|\leq1}|x|^{-b(p+1)}W^{p}W_\tau+\int_{1<|x|\leq\frac{1}{\tau}}|x|^{-b(p+1)}W^{p}W_\tau\\
&&+\int_{\frac{1}{\tau}<|x|}|x|^{-b(p+1)}W^{p}W_\tau.
\end{eqnarray*}
Since $\tau<1$, by \eqref{eq0004}, $W(x)\sim1$ and $W_\tau(x)\sim\tau^{a_c-a}$ in the region $\{x\in\bbr^N\mid|x|\leq1\}$.  It follows that
\begin{eqnarray*}
\int_{|x|\leq1}|x|^{-b(p+1)}W^{p}W_\tau\sim\tau^{a_c-a}\int_0^1 r^{N-1-b(p+1)}\sim\tau^{a_c-a}.
\end{eqnarray*}
Here, we have used the fact that $N-b(p+1)=(p+1)(a_c-a)>0$.  In the region $\{x\in\bbr^N\mid 1<|x|\leq\frac{1}{\tau}\}$, $W_\tau(x)\sim\tau^{a_c-a}$ and $W(x)\sim|x|^{-2(a_c-a)}$ by \eqref{eq0004}.  Thus,
\begin{eqnarray*}
\int_{1<|x|\leq\frac{1}{\tau}}|x|^{-b(p+1)}W^{p}W_\tau\sim\tau^{a_c-a}\int_1^{\frac{1}{\tau}} r^{N-1-b(p+1)-2p(a_c-a)}\sim\tau^{a_c-a},
\end{eqnarray*}
where we have used the fact that $N-b(p+1)-2p(a_c-a)=(1-p)(a_c-a)<0$ and $\tau<1$.  In the region $\{x\in\bbr^N\mid \frac{1}{\tau}<|x|\}$, $W_\tau(x)\sim\tau^{-(a_c-a)}|x|^{-2(a_c-a)}$ and $W(x)\sim|x|^{-2(a_c-a)}$ by \eqref{eq0004}.  Therefore,
\begin{eqnarray*}
\int_{\frac{1}{\tau}<|x|}|x|^{-b(p+1)}W^{p}W_\tau\sim\tau^{-(a_c-a)}\int_{\frac{1}{\tau}}^{+\infty} r^{N-1-b(p+1)-2(p+1)(a_c-a)}\sim\tau^{p(a_c-a)},
\end{eqnarray*}
where we have used the fact that $N-b(p+1)-2(p+1)(a_c-a)=-(p+1)(a_c-a)<0$.  Thus, by $\tau<1$ and $p>1$,
\begin{eqnarray}\label{eqn994}
\langle W, W_{\tau}\rangle_{D^{1,2}_a(\bbr^N)}\sim\tau^{a_c-a}.
\end{eqnarray}
By \eqref{eq0007}, we have $\langle \Psi, \Psi_{s}\rangle_{H^{1}(\mathcal{C})}=\langle W, W_{\tau}\rangle_{D^{1,2}_a(\bbr^N)}$, where $\Psi(t)$ is given by \eqref{eq0026}, $\Psi_s(t)=\Psi(t-s)$ and $s=\ln\tau$.  Then \eqref{eq0025} follows immediately from \eqref{eqn994}.
\end{proof}

\vskip0.12in

Let $v\in H^1(\mathcal{C})$ be nonnegative such that
\begin{eqnarray*}
(\nu-\frac12)(C_{a,b,N}^{-1})^{\frac{p+1}{p-1}}<\|v\|_{H^{1}(\mathcal{C})}^2<(\nu+\frac12)(C_{a,b,N}^{-1})^{\frac{p+1}{p-1}}
\end{eqnarray*}
for some positive integer $\nu\geq2$ and denote
\begin{eqnarray}\label{eq0060}
f:=-\Delta_{\theta}v-\partial_t^2v+(a_c-a)^2v-v^{p}.
\end{eqnarray}
Then $f\in H^{-1}(\mathcal{C})$.  As that in \cite{CFM2017,DSW2021,FG2021}, we consider the following minimizing problem:
\begin{eqnarray*}
d_*^2(v)=\min_{s_j\in\bbr}\|v-\sum_{j=1}^{\nu}\Psi_{s_j}\|^2_{H^1(\mathcal{C})}.
\end{eqnarray*}
By similar arguments as that used in the proof of Proposition~\ref{prop0003}, we can show that $d_*^2(v)$ is attained at some $\{s_j\}\in\bbr^\nu$ and thus, we can write $v=\sum_{j=1}^{\nu}\Psi_{s_j}+\rho$, where $\rho$ satisfies the following orthogonal conditions:
\begin{eqnarray}\label{eq0013}
\langle \Psi_{s_j}', \rho\rangle_{H^{1}(\mathcal{C})}=0\quad\text{for all }j=1,2,\cdots,\nu.
\end{eqnarray}
Clearly, $d_*^2(v)=\|\rho\|_{H^1(\mathcal{C})}^2$.  Moreover, by Proposition~\ref{prop0002}, we know that $d_*(v)\to0$ as $\|f\|_{H^{-1}(\mathcal{C})}\to0$ either for $b_{FS}(a)\leq b<a+1$ with $a<0$ or for $a\leq b<a+1$ with $a\geq0$ and $a+b>0$.  Thus, if $\|f\|_{H^{-1}(\mathcal{C})}\leq\delta$ for $\delta>0$ sufficiently small, we have $\|\rho\|_{H^1(\mathcal{C})}\leq\delta'$ either for $b_{FS}(a)\leq b<a+1$ with $a<0$ or for $a\leq b<a+1$ with $a\geq0$ and $a+b>0$, where $\delta'\to0$ as $\delta\to0$.

\vskip0.12in

Since $\Psi_{s_j}$ are solutions of \eqref{eq0006}, by \eqref{eq0013}, we can rewrite \eqref{eq0060} as follows:
\begin{eqnarray}\label{eq0014}
\left\{\aligned&-\Delta_{\theta}\rho-\partial_t^2\rho+(a_c-a)^2\rho=(\sum_{j=1}^{\nu}\Psi_{s_j}+\rho)^{p}-\sum_{j=1}^{\nu}\Psi_{s_j}^{p}+f,\quad \text{in }\mathcal{C},\\
&\langle \Psi_{s_j}', \rho\rangle_{H^{1}(\mathcal{C})}=0\quad\text{for all }j=1,2,\cdots,\nu.\endaligned\right.
\end{eqnarray}
In what follows, for the sake of simplicity, we denote
\begin{eqnarray*}
R=\min_{i\not=j}|s_i-s_j|\quad\text{and}\quad Q=e^{-(a_c-a)\min_{i\not=j}|s_i-s_j|}
\end{eqnarray*}
as that in \cite{DSW2021}.  Moreover, we also assume that $s_1<s_2<\cdots<s_\nu$ without loss of generality.  For the sake of simplicity, we also denote $s_0=-\infty$ and $s_{\nu+1}=+\infty$.
\begin{lemma}\label{lemn0001}
Let $b_{FS}(a)\leq b<a+1$ for $a<0$ and $a\leq b<a+1$ for $a\geq0$ and $a+b>0$.  Then
\begin{eqnarray}\label{eq1146}
\|f\|_{H^{-1}(\mathcal{C})}\gtrsim Q+O(Q^{\frac{1}{2}}\|\rho\|_{H^1(\mathcal{C})}+\|\rho\|_{H^1(\mathcal{C})}^{p+1})
\end{eqnarray}
for $\|f\|_{H^{-1}(\mathcal{C})}\leq\delta$ with $\delta>0$ sufficiently small.
\end{lemma}
\begin{proof}
Suppose that $R=s_{j_0+1}-s_{j_0}$ for some $j_0\in\{1, 2,\cdots, \nu-1\}$.  Multiplying \eqref{eq0014} with $-\Psi_{s_{j_0}}'$ and integrating by parts, we have
\begin{eqnarray*}
\|f\|_{H^{-1}(\mathcal{C})}\gtrsim\langle (\sum_{j=1}^{\nu}\Psi_{s_j}+\rho)^{p}-\sum_{j=1}^{\nu}\Psi_{s_j}^{p}, -\Psi_{s_{j_0}}'\rangle_{L^2(\mathcal{C})}.
\end{eqnarray*}
In the region $\{3|\rho|<\sum_{j=1}^{\nu}\Psi_{s_j}\}$, by the Taylor expansion,
\begin{eqnarray}
(\sum_{j=1}^{\nu}\Psi_{s_j}+\rho)^{p}-(\sum_{j=1}^{\nu}\Psi_{s_j})^{p}-p(\sum_{j=1}^{\nu}\Psi_{s_j})^{p-1}\rho&=&p(p-1)(\sum_{j=1}^{\nu}\Psi_{s_j}+\xi\rho)^{p-2}\rho^2\notag\\
&\sim&(\sum_{j=1}^{\nu}\Psi_{s_j})^{p-2}\rho^2>0,\label{eqn19985}
\end{eqnarray}
where $\xi\in(0, 1)$.
In the region $\{\sum_{j=1}^{\nu}\Psi_{s_j}\leq 3|\rho|\}$, since $\Psi_{s_j}$ are all positive, we also have $\Psi_{s_{j_0}}\leq3|\rho|$ which, together with $-\Psi_{s_{j_0}}'\sim\Psi_{s_{j_0}}$, implies that
\begin{eqnarray}\label{eqn19984}
\bigg((\sum_{j=1}^{\nu}\Psi_{s_j}+\rho)^{p}-(\sum_{j=1}^{\nu}\Psi_{s_j})^{p}-p(\sum_{j=1}^{\nu}\Psi_{s_j})^{p-1}\rho\bigg)\Psi_{s_{j_0}}'\lesssim\rho^{p+1}.
\end{eqnarray}
By Proposition~\ref{prop0002} and Lemma~\ref{lem0001}, $R\to+\infty$ as $\delta\to0$.  Thus,
\begin{eqnarray*}
e^{-(a_c-a)|t-s_i|}\gtrsim e^{-(a_c-a)|t-s_{i-1}|}+e^{-(a_c-a)|t-s_{i+1}|}\sim \sum_{j\not=i}e^{-(a_c-a)|t-s_j|}
\end{eqnarray*}
in $(s_i-\frac{R}{2}+O(1),s_i+\frac{R}{2}+O(1))$ for all $i=1,2,\cdots,\nu$, which, together with \eqref{eq0026}, implies that $\{\Psi_{s_i}\gtrsim\sum_{j\not=i}\Psi_{s_j}\}$ in the region $(s_i-\frac{R}{2}+O(1),s_i+\frac{R}{2}+O(1))$ for all $i=1,2,\cdots,\nu$.  It follows from the Taylor expansion that
\begin{eqnarray}\label{eqn19999}
(\sum_{j=1}^{\nu}\Psi_{s_j})^{p}-\sum_{j=1}^{\nu}\Psi_{s_j}^{p}=p(\Psi_{s_i}+\xi_i\sum_{j\not=i}\Psi_{s_j})^{p-1}\sum_{j\not=i}\Psi_{s_j}\sim\Psi_{s_i}^{p-1}\sum_{j\not=i}\Psi_{s_j}
\end{eqnarray}
in the region $(s_i-\frac{R}{2}+O(1),s_i+\frac{R}{2}+O(1))$ for all $i=1,2,\cdots,\nu$, where $\xi_i\in(0, 1)$.  In the region $\bbr\backslash(\cup_{i=1}^{\nu}(s_i-\frac{R}{2}+O(1),s_i+\frac{R}{2}+O(1)))$, by \eqref{eq0026},
\begin{eqnarray}\label{eqn19998}
|(\sum_{j=1}^{\nu}\Psi_{s_j})^{p}-\sum_{j=1}^{\nu}\Psi_{s_j}^{p}|\lesssim \sum_{j=1}^{\nu}\Psi_{s_j}^{p}\sim \sum_{j=1}^{\nu}e^{-p(a_c-a)|t-s_j|}.
\end{eqnarray}
Thus, by $-\Psi_{s_{j_0}}'\sim\Psi_{s_{j_0}}$ once more, \eqref{eqn19998}, $p>1$ and the orthogonal conditions in \eqref{eq0014},
\begin{eqnarray*}
\|f\|_{H^{-1}(\mathcal{C})}&\gtrsim&-\int_{\mathcal{C}}((\sum_{j=1}^{\nu}\Psi_{s_j})^{p}-\sum_{j=1}^{\nu}\Psi_{s_j}^{p}+p((\sum_{j=1}^{\nu}\Psi_{s_j})^{p-1}-\Psi_{s_{j_0}}^{p-1})\rho)\Psi_{s_{j_0}}')\notag\\
&&+O(\|\rho\|_{H^1(\mathcal{C})}^{p+1})\notag\\
&\gtrsim&\sum_{i=1}^\nu\int_{(s_i-\frac{R}{2}+O(1),s_i+\frac{R}{2}+O(1))}\Psi_{s_i}^{p-1}\Psi_{s_{j_0}}\sum_{j\not=i}^{\nu}\Psi_{s_j}+o(Q)\notag\\
&&-\int_{(s_{j_0}-\frac{R}{2}+O(1),s_{j_0}+\frac{R}{2}+O(1))}\Psi_{s_{j_0}}^{p-1}\sum_{j\not=j_0}\Psi_{s_j}|\rho|\notag\\
&&-\int_{(s_{j_0}-\frac{R}{2}+O(1),s_{j_0}+\frac{R}{2}+O(1))^c}\Psi_{s_{j_0}}\sum_{j=1}^{\nu}\Psi_{s_j}^{p-1}|\rho|+O(\|\rho\|_{H^1(\mathcal{C})}^{p+1})\\
&\sim&\sum_{i=1}^{\nu}\int_{(s_i-\frac{R}{2}+O(1),s_i+\frac{R}{2}+O(1))}\Psi_{s_i}^{p-1}\Psi_{s_{j_0}}\sum_{j\not=i}^{\nu}\Psi_{s_j}+o(Q)\\
&&+O(Q^{\frac{1}{2}}\|\rho\|_{H^1(\mathcal{C})}+\|\rho\|_{H^1(\mathcal{C})}^{p+1}).
\end{eqnarray*}
By \eqref{eq0026}, $p>1$ and \cite[Lemma~4.1]{WW2020},
\begin{eqnarray*}
&&\sum_{i=1}^{\nu}\int_{(s_i-\frac{R}{2}+O(1),s_i+\frac{R}{2}+O(1))}\Psi_{s_i}^{p-1}\Psi_{s_{j_0}}\sum_{j\not=i}^{\nu}\Psi_{s_j}\\
&\sim&
\int_{(s_{j_0}-\frac{R}{2}+O(1),s_{j_0}+\frac{R}{2}+O(1))}\Psi_{s_{j_0}}^{p}(\Psi_{s_{j_0+1}}+\Psi_{s_{j_0-1}})\\
&&+\int_{(s_{j_0+1}-\frac{R}{2}+O(1),s_{j_0+1}+\frac{R}{2}+O(1))}\Psi_{s_{j_0}}^{2}\Psi_{s_{j_0+1}}^{p-1}\\
&&+\int_{(s_{j_0-1}-\frac{R}{2}+O(1),s_{j_0-1}+\frac{R}{2}+O(1))}\Psi_{s_{j_0}}^{2}\Psi_{s_{j_0-1}}^{p-1}+o(Q)\\
&\sim& \int_0^{\frac{R}{2}}e^{-(a_c-a)pr}e^{-(a_c-a)(R-r)}+\int_0^{\frac{R}{2}}e^{-2(a_c-a)(R-r)}e^{-(a_c-a)(p-1)r}+o(Q)\\
&\sim&Q.
\end{eqnarray*}
It follows that \eqref{eq1146} holds for $\delta>0$ sufficiently small.
\end{proof}

\vskip0.12in

As that in \cite{DSW2021}, we want to drive the precise behavior of first approximation of $\rho$ by considering the following equation:
\begin{eqnarray}\label{eq0015}
\left\{\aligned&-\Delta_{\theta}\phi-\partial_t^2\phi+(a_c-a)^2\phi\\
&=|\sum_{j=1}^{\nu}\Psi_{s_j}+\phi|^{p-1}(\sum_{j=1}^{\nu}\Psi_{s_j}+\phi)\\
&-\sum_{j=1}^{\nu}\Psi_{s_j}^{p}+\sum_{j=1}^{\nu}c_j\Psi_{s_j}^{p-1}\Psi_{s_j}',\quad \text{in }\mathcal{C},\\
&\langle \Psi_{s_j}', \phi\rangle_{H^{1}(\mathcal{C})}=0\quad\text{for all }j=1,2,\cdots,\nu,\endaligned\right.
\end{eqnarray}
where $c_j$ and $\phi$ are all unknowns.  By \eqref{eq0026} and some elementary inequalities, we can rewrite
\begin{eqnarray*}
|\sum_{j=1}^{\nu}\Psi_{s_j}+\phi|^{p-1}(\sum_{j=1}^{\nu}\Psi_{s_j}+\phi)-\sum_{j=1}^{\nu}\Psi_{s_j}^{p}&=&p(\sum_{j=1}^{\nu}\Psi_{s_j})^{p-1}\phi+(\sum_{j=1}^{\nu}\Psi_{s_j})^{p}\\
&&-\sum_{j=1}^{\nu}\Psi_{s_j}^{p}+O(|\phi|^{\sigma_p}),
\end{eqnarray*}
where $\sigma_p=2$ for $p\geq2$ and $\sigma_p=p$ for $1<p<2$.
Thus, \eqref{eq0015} can be rewritten as follows:
\begin{eqnarray}\label{eq0020}
\left\{\aligned&\mathcal{L}(\phi)=E+N(\phi)+\sum_{j=1}^{\nu}c_j\Psi_{s_j}^{p-1}\Psi_{s_j}',\quad \text{in }\mathcal{C},\\
&\langle \Psi_{s_j}', \phi\rangle_{H^{1}(\mathcal{C})}=0\quad\text{for all }j=1,2,\cdots,\nu,\endaligned\right.
\end{eqnarray}
where the linear operator $\mathcal{L}(\phi)$ is given by
\begin{eqnarray}\label{eq0062}
\mathcal{L}(\phi):=-\Delta_{\theta}\phi-\partial_t^2\phi+(a_c-a)^2\phi-p(\sum_{j=1}^{\nu}\Psi_{s_j})^{p-1}\phi,
\end{eqnarray}
$E=(\sum_{j=1}^{\nu}\Psi_{s_j})^{p}-\sum_{j=1}^{\nu}\Psi_{s_j}^{p}$ is the error and $N(\phi)=O(|\phi|^{\sigma_p})$ is the nonlinear part.
\begin{lemma}\label{lemn0002}
For $\delta>0$ sufficiently small, we have
\begin{eqnarray}\label{eqn19997}
\|E\|_{\natural}:=\sum_{i=1}^{\nu}\sup_{t\in(\frac{s_i+s_{i-1}}{2}, \frac{s_{i+1}+s_i}{2})}\frac{|E|}{Qe^{-(a_c-a)(p-2)|t-s_i|}}\lesssim1
\end{eqnarray}
for $1<p<3$ and
\begin{eqnarray}\label{eqn19996}
\|E\|_{\sharp}:=\sum_{i=1}^{\nu}\sup_{t\in(\frac{s_i+s_{i-1}}{2}, \frac{s_{i+1}+s_i}{2})}\frac{|E|}{Qe^{-(1-\varsigma)(a_c-a)|t-s_i|}}\lesssim1
\end{eqnarray}
for $p\geq3$ with $\varsigma>0$ sufficiently small.
\end{lemma}
\begin{proof}
By \eqref{eq0026} and similar arguments as that used for \eqref{eqn19999},
\begin{eqnarray}
E&\sim&\Psi_{s_i}^{p-1}\Psi_{s_{i+1}}\notag\\
&\sim&e^{-(p-1)(a_c-a)|t-s_i|}e^{-(a_c-a)|t-s_{i+1}|}\notag\\
&\sim& e^{-(a_c-a)|s_{i}-s_{i+1}|}e^{-(a_c-a)(p-2)|t-s_i|}\label{eqn19990}
\end{eqnarray}
in the region $(s_i, \frac{s_{i+1}+s_i}{2})$ for all $i=1,2,\cdots,\nu-1$ and
\begin{eqnarray}
E&\sim&\Psi_{s_i}^{p-1}\Psi_{s_{i-1}}\notag\\
&\sim&e^{-(p-1)(a_c-a)|t-s_i|}e^{-(a_c-a)|t-s_{i-1}|}\notag\\
&\sim& e^{-(a_c-a)|s_{i}-s_{i-1}|}e^{-(a_c-a)(p-2)|t-s_i|}\label{eqn19989}
\end{eqnarray}
in the region $(\frac{s_{i-1}+s_i}{2}, s_i)$ for all $i=2,3,\cdots,\nu$.  In the region $(-\infty, s_1)$, since $s_1<s_2<\cdots<s_{\nu-1}<s_\nu$,
\begin{eqnarray}
E&\sim&\Psi_{s_1}^{p-1}\Psi_{s_{2}}\notag\\
&\sim&e^{-(p-1)(a_c-a)|t-s_1|}e^{-(a_c-a)|t-s_{2}|}\notag\\
&\sim& e^{-(a_c-a)|s_{1}-s_{2}|}e^{-p(a_c-a)|t-s_1|}.\label{eqn19970}
\end{eqnarray}
In the region $(s_{\nu}, +\infty)$, since $s_1<s_2<\cdots<s_{\nu-1}<s_\nu$,
\begin{eqnarray}
E&\sim&\Psi_{s_\nu}^{p-1}\Psi_{s_{\nu-1}}\notag\\
&\sim&e^{-(p-1)(a_c-a)|t-s_\nu|}e^{-(a_c-a)|t-s_{\nu-1}|}\notag\\
&\sim& e^{-(a_c-a)|s_{\nu}-s_{\nu-1}|}e^{-p(a_c-a)|t-s_i|}.\label{eqn19969}
\end{eqnarray}
\eqref{eqn19997} and \eqref{eqn19996} then follow immediately from \eqref{eqn19990}, \eqref{eqn19989} and \eqref{eqn19970}, \eqref{eqn19969}.
\end{proof}

To solve \eqref{eq0015}, we shall use the fix point argument, which leads us to establish a good linear theory by considering the following linear equation:
\begin{eqnarray*}\label{eq0021}
\left\{\aligned&\mathcal{L}(\phi)=g,\quad \text{in }\mathcal{C},\\
&\langle \Psi_{s_j}', \phi\rangle_{H^{1}(\mathcal{C})}=0\quad\text{for all }j=1,2,\cdots,\nu,\endaligned\right.
\end{eqnarray*}
where $g$ satisfies $\langle \Psi_{s_j}', g\rangle_{L^{2}(\mathcal{C})}=0$ for all $j=1,2,\cdots,\nu$ and $\mathcal{L}(\phi)$ is given by \eqref{eq0062}.  Based on Lemma~\ref{lemn0002}, we shall introduce the following spaces:
\begin{eqnarray*}
X=\{\phi\in H^1(\mathcal{C})\mid \|\phi\|_{\natural}<+\infty\}\quad, Y=\{\phi\in L^2(\mathcal{C})\mid \|\phi\|_{\natural}<+\infty\},
\end{eqnarray*}
and
\begin{eqnarray*}
\widehat{X}=\{\phi\in H^1(\mathcal{C})\mid \|\phi\|_{\sharp}<+\infty\}\quad, \widehat{Y}=\{\phi\in L^2(\mathcal{C})\mid \|\phi\|_{\sharp}<+\infty\}.
\end{eqnarray*}
Clearly, $X$, $Y$ and $\widehat{X}$, $\widehat{Y}$ are all Banach spaces.  Let
\begin{eqnarray*}
X^{\perp}&=&\{\phi\in X\mid \langle \Psi_{s_j}', \phi\rangle_{H^{1}(\mathcal{C})}=0\quad\text{for all }j=1,2,\cdots,\nu\},\\
Y^{\perp}&=&\{\phi\in Y\mid \langle \Psi_{s_j}', \phi\rangle_{L^{2}(\mathcal{C})}=0\quad\text{for all }j=1,2,\cdots,\nu\}
\end{eqnarray*}
and
\begin{eqnarray*}
\widehat{X}^{\perp}&=&\{\phi\in \widehat{X}\mid \langle \Psi_{s_j}', \phi\rangle_{H^{1}(\mathcal{C})}=0\quad\text{for all }j=1,2,\cdots,\nu\},\\
\widehat{Y}^{\perp}&=&\{\phi\in \widehat{Y}\mid \langle \Psi_{s_j}', \phi\rangle_{L^{2}(\mathcal{C})}=0\quad\text{for all }j=1,2,\cdots,\nu\},
\end{eqnarray*}
then we have the following.
\begin{lemma}\label{lem0002}
Let $b_{FS}(a)\leq b<a+1$ for $a<0$ and $a\leq b<a+1$ for $a\geq0$ and $a+b>0$.
\begin{enumerate}
\item[$(1)$]\quad If $p\geq3$, then for $\delta>0$ sufficiently small, there exists a unique $\phi\in \widehat{X}^{\perp}$ such that $\mathcal{L}(\phi)=g$ and $\|\phi\|_{\sharp}\lesssim\|g\|_{\sharp}$ for every $g\in \widehat{Y}^{\perp}$.
\item[$(2)$]\quad If $1<p<3$, then for $\delta>0$ sufficiently small, there exists a unique $\phi\in X^{\perp}$ such that $\mathcal{L}(\phi)=g$ and $\|\phi\|_{\natural}\lesssim\|g\|_{\natural}$ for every $g\in Y^{\perp}$.
\end{enumerate}
Here $\mathcal{L}(\phi)$ is given by \eqref{eq0062}.
\end{lemma}
\begin{proof}
Since the proof is rather standard nowadays (cf. \cite{PDM2003,PFM2003,PMPP2011,WY2007,WY2010}), we only sketch it here.  We start by proving the a-priori estimates $\|\phi\|_{\sharp}\lesssim\|g\|_{\sharp}$ for $p\geq3$ and
$\|\phi\|_{\natural}\lesssim\|g\|_{\natural}$ for $1<p<3$.  Assuming the contrary, that is, there exist $\{g_{n}\}$ and $\{\delta_n\}$ such that $\|g_{n}\|_{\sharp}\to0$ and $\delta_n\to0$ as $n\to\infty$ and $\|\phi_{n}\|_{\sharp}=1$ for $p\geq3$ while, $\|g_{n}\|_{\natural}\to0$ and $\delta_n\to0$ as $n\to\infty$ and $\|\phi_{n}\|_{\natural}=1$ for $1<p<3$.  Since $\delta_n\to0$ as $n\to\infty$, by proposition~\ref{prop0002},
\begin{eqnarray*}
R_n=\min_{i\not=j}|s_{i,n}-s_{j,n}|\to+\infty\quad\text{as }n\to\infty.
\end{eqnarray*}
By the definition of the norms $\|\cdot\|_{\natural}$ and $\|\cdot\|_{\sharp}$ given by \eqref{eqn19997} and \eqref{eqn19996},
\begin{eqnarray*}
|\mathcal{L}(\phi_n)|\lesssim\left\{\aligned&\|g_{n}\|_{\natural}\sum_{i=1}^{\nu}Q_n\varphi_{i,n}(t)\chi_{i,n}(t),\quad 1<p<3,\\
&\|g_{n}\|_{\sharp}\sum_{i=1}^{\nu}Q_n\varphi_{i,n}(t)\chi_{i,n}(t),\quad p\geq3,
\endaligned\right.
\end{eqnarray*}
where
\begin{eqnarray}\label{eqn19993}
\varphi_{i,n}(t)=\left\{\aligned&e^{-(1-\varsigma)(a_c-a)|t-s_{i,n}|},\quad p\geq3,\\
&e^{-(p-2)(a_c-a)|t-s_{i,n}|},\quad 1<p<3,\endaligned\right.
\end{eqnarray}
and $\chi_{i,n}$ is a cut-off function such that
\begin{eqnarray}\label{eqn19992}
\chi_{i,n}(t)=\left\{\aligned&1,\quad t\in (\frac{s_{i,n}+s_{i-1,n}}{2}, \frac{s_{i+1,n}+s_{i,n}}{2}),\\
&0,\quad t\in (\frac{s_{i,n}+s_{i-1,n}}{2}, \frac{s_{i+1,n}+s_{i,n}}{2})^c.\endaligned\right.
\end{eqnarray}
By \eqref{eq0026}, it is easy to see that
\begin{eqnarray*}
\mathcal{L}(\varphi_{i,n})\gtrsim \left\{\aligned&e^{-(1-\varsigma)(a_c-a)|t-s_{i,n}|},\quad p\geq3,\\
&e^{-(p-2)(a_c-a)|t-s_{i,n}|},\quad 1<p<3\endaligned\right.
\end{eqnarray*}
in $(\frac{s_{i,n}+s_{i-1,n}}{2}, \frac{s_{i+1,n}+s_{i,n}}{2})\backslash(s_i-T,s_i+T)$ for a sufficiently large $T>0$.  Thus, by the maximum principle,
\begin{eqnarray}\label{eqn19995}
|\phi_n|\lesssim\left\{\aligned&\|g_{n}\|_{\natural}Q_n\varphi_{i,n}(t),\quad 1<p<3,\\
&\|g_{n}\|_{\sharp}Q_n\varphi_{i,n}(t),\quad p\geq3
\endaligned\right.
\end{eqnarray}
in $(\frac{s_{i,n}+s_{i-1,n}}{2}, \frac{s_{i+1,n}+s_{i,n}}{2})\backslash(s_{i,n}-T, s_{i,n}+T)$ for all $i=1,2,\cdots,\nu$.  On the other hand, by $\|\phi_{n}\|_{\sharp}=1$ and $\|g_{n}\|_{\sharp}=o_n(1)$ for $p\geq3$ while
$\|\phi_{n}\|_{\natural}=1$ and $\|g_{n}\|_{\natural}=o_n(1)$ for $1<p<3$, it is standard to use the Moser iteration and the Sobolev embedding theorem to show that $Q_n^{-1}\phi_{n}(\cdot+s_{i,n})\to\widehat{\phi}$ uniformly in every compact set of $\mathcal{C}$ as $n\to\infty$ for all $i=1,2,\cdots,\nu$, where $\widehat{\phi}$ is a solution of \eqref{eq0016}.  We recall that by the nondegeneracy of $\Psi$ in $H^1(\mathcal{C})$, $\Psi'$ is the only nonzero solution of \eqref{eq0016}.  Thus, we must have that $\widehat{\phi}=C\Psi'$, which together with the orthogonal condition in $X^{\perp}$ for $1<p<3$ and the orthogonal condition in $\widehat{X}^{\perp}$ for $p\geq3$, implies that $\widehat{\phi}=0$.  Since $\varphi_{i,n}(t)\sim1$ in $[s_{i,n}-T, s_{i,n}+T]$ for fixed $T>0$, $\frac{|\phi_{n}|}{Q_n\varphi_{i,n}(t)}=o_n(1)$ in $[s_{i,n}-T, s_{i,n}+T]$ for fixed $T>0$.  Thus, by \eqref{eqn19995}, $\|\phi_{n}\|_{\sharp}=o_n(1)$ for $p\geq3$ and $\|\phi_{n}\|_{\natural}=o_n(1)$ for $1<p<3$.  It contradicts $\|\phi_{n}\|_{\sharp}=1$ for $p\geq3$ and
$\|\phi_{n}\|_{\natural}=1$ for $1<p<3$.
The a-priori estimates $\|\phi\|_{\sharp}\lesssim\|g\|_{\sharp}$ for $p\geq3$ and
$\|\phi\|_{\natural}\lesssim\|g\|_{\natural}$ for $1<p<3$ imply that $\mathcal{L}: X^{\perp}\to Y^{\perp}$ for $1<p<3$ and $\mathcal{L}: \widehat{X}^{\perp}\to \widehat{Y}^{\perp}$ for $p\geq3$ are injective for $\delta>0$ sufficiently small.  Since $(\sum_{j=1}^{\nu}\Psi_{s_j})^{p-1}\to0$ as $|t|\to+\infty$ by \eqref{eq0026}, it is standard to use the the Fredholm alternative to show that for $\delta>0$ sufficiently small, $\mathcal{L}(\phi)=g$ is unique solvable in $X^{\perp}$ for every $g\in Y^{\perp}$ in the case of $1<p<3$ and $\mathcal{L}(\phi)=g$ is unique solvable in $\widehat{X}^{\perp}$ for every $g\in \widehat{Y}^{\perp}$ in the case of $p\geq3$.
\end{proof}

Let us go back to \eqref{eq0020}, then we have the following.
\begin{lemma}\label{lem0003}
Let $b_{FS}(a)\leq b<a+1$ for $a<0$ and $a\leq b<a+1$ for $a\geq0$ and $a+b>0$.  Then \eqref{eq0020} has a unique solution $(\psi,c_1,c_2,\cdots,c_\nu)$ for $\delta>0$ sufficiently small.  Moreover,
\begin{eqnarray}\label{eq0027}
\|\phi\|_{H^1(\mathcal{C})}\lesssim \left\{\aligned &Q,\quad p>2,\\
&Q|\log(Q)|^{\frac12},\quad p=2,\\
&Q^{\frac{p}{2}},\quad 1<p<2.
\endaligned\right.
\end{eqnarray}
and $\sum_{j=1}^\nu|c_l|\lesssim Q$.
\end{lemma}
\begin{proof}
Since $R\to+\infty$ as $\delta\to0$ and $p>1$, by \cite[Lemma~6]{LW05} and \eqref{eq0026},
\begin{eqnarray}\label{eq0033}
\langle\Psi_{s_j}^{p-1}\Psi_{s_j}',\Psi_{s_i}'\rangle_{L^2(\mathcal{C})}\sim Q
\end{eqnarray}
for $i\not=j$ for $\delta>0$ sufficiently small.  Thus, $\{c_j\}$ in \eqref{eq0020} can be chosen to be the unique solution of the following equation:
\begin{eqnarray*}
(\langle\Psi_{s_j}^{p-1}\Psi_{s_j}',\Psi_{s_i}'\rangle_{L^2(\mathcal{C})})_{i,j=1,2,\cdots,\nu}\bullet(c_j)_{j=1,2,\cdots,\nu}=-(\langle E+N(\phi), \Psi_{s_i}'\rangle_{L^2(\mathcal{C})})_{i=1,2,\cdots,\nu}.
\end{eqnarray*}
By Lemmas~\ref{lemn0002} and \ref{lem0002}, and adapting the fix point arguments in a standard way (cf. \cite{PDM2003,PFM2003,PMPP2011,WY2007,WY2010}), \eqref{eq0020} is unique solvable in the set $\widehat{B}=\{\phi\in \widehat{X}^{\perp}\mid \|\phi\|_{\sharp}\leq C\}$
in the case of $p\geq3$ and in the set $B=\{\phi\in X^{\perp}\mid \|\phi\|_{\natural}\leq C\}$
in the case of $1<p<3$ for a sufficiently large $C>0$.  Note that $N(\phi)=O(\phi^{2})$ for $p\geq2$, by $-\Psi'\sim\Psi$ and \cite[Lemma~6]{LW05},
\begin{eqnarray*}
|\langle N(\phi), \Psi_{s_j}'\rangle|\lesssim\int_{\mathcal{C}}\sum_{i=1}^{\nu}(Q\varphi_{i}(t)\chi_{i}(t))^{2}\Psi_{s_j}\lesssim Q^2,
\end{eqnarray*}
where $\varphi_i$ is given by \eqref{eqn19993} and $\chi_{i}$ is a cut-off function given by \eqref{eqn19992}.  For $1<p<2$, $N(\phi)=O(|\phi|^{p})$.  Thus, by $-\Psi'\sim\Psi$ and \eqref{eq0026},
\begin{eqnarray*}
|\langle N(\phi), \Psi_{s_j}'\rangle|\lesssim\int_{\mathcal{C}}\sum_{i=1}^{\nu}(Q\varphi_{i}(t)\chi_{i}(t))^{p}\Psi_{s_j}\sim Q^p\int_0^{\frac{R}{2}}e^{(2-p)p(a_c-a)r-r}=O(Q^{\frac{p^2+1}{2}}).
\end{eqnarray*}
On the other hand, by \eqref{eq0026}, \eqref{eqn19990}, \eqref{eqn19989}, \eqref{eqn19970}, \eqref{eqn19969} and $-\Psi_{s_j}'\sim\Psi_{s_j}$,
\begin{eqnarray*}\label{eqn1001}
|\langle E, \Psi_{s_j}'\rangle_{L^2(\mathcal{C})}|\sim Q\sum_{i=1}^{\nu}\int_{(\frac{s_{i-1}+s_i}{2}, \frac{s_i+s_{i+1}}{2})}\Psi_{s_i}^{p-2}\Psi_{s_j}\sim Q\int_{(\frac{s_{j-1}+s_j}{2}, \frac{s_j+s_{j+1}}{2})}\Psi_{s_j}^{p-1}\sim Q.
\end{eqnarray*}
It follows from $p>1$ that $\sum_{j=1}^\nu|c_l|\lesssim Q$ and
\begin{eqnarray}\label{eqn19988}
\|\phi\|_{H^1(\mathcal{C})}^2&=&\langle E+N(\phi)+\sum_{j=1}^{\nu}c_j\Psi_{s_j}^{p-1}\Psi_{s_j}', \phi\rangle_{L^2(\mathcal{C})}\notag\\
&\leq& \langle E, \phi\rangle_{L^2(\mathcal{C})}+O(Q\|\phi\|_{H^1(\mathcal{C})}+\|\phi\|_{H^1(\mathcal{C})}^{\sigma_p+1}),
\end{eqnarray}
where $\sigma_p=2$ for $p\geq2$ and $\sigma_p=p$ for $1<p<2$.
Since $\|\phi\|_{\sharp}\leq C$ for $p\geq3$, by \eqref{eqn19996},
\begin{eqnarray*}
\langle E, \phi\rangle_{L^2(\mathcal{C})}\lesssim\sum_{i=1}^{\nu}Q^2\int_{(\frac{s_{i-1}+s_i}{2}, \frac{s_i+s_{i+1}}{2})}\Psi_{s_i}^{(2-2\varsigma)}\sim Q^2
\end{eqnarray*}
for $p\geq3$.
For $1<p<3$, $\|\phi\|_{\natural}\leq C$.  Thus, by \eqref{eqn19997},
\begin{eqnarray*}
\langle E, \phi\rangle_{L^2(\mathcal{C})}\lesssim\sum_{i=1}^{\nu}Q^2\int_{(\frac{s_{i-1}+s_i}{2}, \frac{s_i+s_{i+1}}{2})}\Psi_{s_i}^{2(p-2)}\sim \left\{\aligned&Q^2,\quad p>2,\\
&Q^2\log(Q),\quad p=2,\\
&Q^{p},\quad 1<p<2.\endaligned\right.
\end{eqnarray*}
\eqref{eq0027} then follows from \eqref{eqn19988}.
\end{proof}

Let $\varphi=\rho-\phi$, then by \eqref{eq0014} and \eqref{eq0015},
\begin{eqnarray}\label{eq0028}
\left\{\aligned&-\Delta_{\theta}\varphi-\partial_t^2\varphi+(a_c-a)^2\varphi\\
&=(\sum_{j=1}^{\nu}\Psi_{s_j}+\phi+\varphi)^{p}-|\sum_{j=1}^{\nu}\Psi_{s_j}+\phi|^{p-1}(\sum_{j=1}^{\nu}\Psi_{s_j}+\phi)\\
&-\sum_{j=1}^{\nu}c_j\Psi_{s_j}^{p-1}\Psi_{s_j}'+f,\quad \text{in }\mathcal{C},\\
&\langle \Psi_{s_j}', \varphi\rangle_{H^{1}(\mathcal{C})}=0\quad\text{for all }j=1,2,\cdots,\nu.\endaligned\right.
\end{eqnarray}
Let $M_0=\text{span}\{\Psi_{s_j}\}$ and $M=\text{span}\{\Psi_{s_j}'\}$.  Then by the orthogonal conditions satisfied by $\varphi$, we can decompose $\varphi=\sum_{j=1}^{\nu}\beta_j\Psi_{s_j}+\Psi^{\perp}$,
where $\Psi^{\perp}\in(M_0\oplus M)^{\perp}$ in $H^1(\mathcal{C})$.
\begin{lemma}\label{lem0004}
Let $b_{FS}(a)\leq b<a+1$ for $a<0$ and $a\leq b<a+1$ for $a\geq0$ and $a+b>0$.  Then for $\delta>0$ sufficiently small, we have
\begin{eqnarray}\label{eq0035}
|\beta_{j}|\lesssim \|f\|_{H^{-1}(\mathcal{C})}+Q^2
\end{eqnarray}
and
\begin{eqnarray}\label{eq0036}
\|\Psi^{\perp}\|_{H^1(\mathcal{C})}\lesssim \|f\|_{H^{-1}(\mathcal{C})}+Q^2.
\end{eqnarray}
\end{lemma}
\begin{proof}
Since $\Psi$ is the minimizer of \eqref{eq0009}, the Morse index of $\Psi$ is equal to $1$.  It follows from the nondegeneracy of $\Psi$ that
\begin{eqnarray}\label{eq0029}
\int_{\mathcal{C}}|\nabla_\theta v|^2+|\partial_t v|^2+(a_c-a)^2v^2>(p+2\ve)\int_{\mathcal{C}}\Psi^{p-1}v^2
\end{eqnarray}
for all $v\in\text{span}\{\Psi,\Psi'\}^{\perp}$ with some $\ve>0$ sufficiently small.  Since $R\to+\infty$ as $\delta\to0$ by Proposition~\ref{prop0002} and $p>1$, for $\delta>0$ sufficiently small, it is standard to use \eqref{eq0029} and the exponential decay of $\Psi$ at infinity given by \eqref{eq0026} to show that
\begin{eqnarray}\label{eq0032}
\int_{\mathcal{C}}|\nabla_\theta v|^2+|\partial_t v|^2+(a_c-a)^2v^2>(p+\ve)\int_{\mathcal{C}}(\sum_{j=1}^{\nu}\Psi_{s_j})^{p-1}v^2
\end{eqnarray}
for all $v\in(M_0\oplus M)^{\perp}$.  By \eqref{eq0026}, $\|\phi\|_{\sharp}\leq C$ for $p\geq3$ and $\|\phi\|_{\natural}\leq C$ for $1<p<3$,
\begin{eqnarray*}
&&(\sum_{j=1}^{\nu}\Psi_{s_j}+\phi+\varphi)^{p}-|\sum_{j=1}^{\nu}\Psi_{s_j}+\phi|^{p-1}(\sum_{j=1}^{\nu}\Psi_{s_j}+\phi)\\
&=&p|\sum_{j=1}^{\nu}\Psi_{s_j}+\phi|^{p-1}\varphi+O(|\varphi|^{\sigma_p})\\
&=&p(\sum_{j=1}^{\nu}\Psi_{s_j})^{p-1}\varphi+q(t,\theta)\varphi+O(|\varphi|^{\sigma_p}),
\end{eqnarray*}
where $\|q\|_{L^{\infty}(\mathcal{C})}\leq\widehat{\delta}$ with $\widehat{\delta}\to0$ as $\delta\to0$, $\sigma_p=2$ for $p\geq2$ and $\sigma_p=p$ for $1<p<2$.  Now, multiplying \eqref{eq0028} with $\Psi_{s_j}$ for all $j$ and $\Psi^{\perp}$, respectively, and integrating by parts,
\begin{eqnarray*}
(1-p)\beta_{j}\|\Psi\|_{H^1(\mathcal{C})}^2&=&\int_{\mathcal{C}}(p(\sum_{l=1}^{\nu}\Psi_{s_l})^{p-1}+q(t,\theta))(\sum_{i\not=j}\beta_{i}\Psi_{s_i}+\Psi^{\perp})\Psi_{s_j}\\
&&+\beta_{j}\int_{\mathcal{C}}(p(\sum_{l=1}^{\nu}\Psi_{s_l})^{p-1}-p\Psi_{s_j}^{p-1}+q(t,\theta))\Psi_{s_j}^2\\
&&+\langle f, \Psi_{s_j}\rangle_{L^2(\mathcal{C})}+\sum_{i\not=j}c_i\langle\Psi_{s_i}^{p-1}\Psi_{s_i}', \Psi_{s_j}\rangle_{L^2(\mathcal{C})}\\
&&+O(\sum_{i=1}^{\nu}\beta_{i}^{\sigma_p+1}+\|\Psi^{\perp}\|_{H^1(\mathcal{C})}^{\sigma_p+1})
\end{eqnarray*}
and
\begin{eqnarray*}
\|\Psi^{\perp}\|_{H^1(\mathcal{C})}^2&=&\int_{\mathcal{C}}(p(\sum_{l=1}^{\nu}\Psi_{s_l})^{p-1}+q(t,\theta))(\sum_{i=1}^{\nu}\beta_{i}\Psi_{s_i}+\Psi^{\perp})\Psi^{\perp}\\
&&+\langle f, \Psi^{\perp}\rangle_{L^2(\mathcal{C})}+O(\sum_{i=1}^{\nu}\beta_{i}^{\sigma_p+1}+\|\Psi^{\perp}\|_{H^1(\mathcal{C})}^{\sigma_p+1}).
\end{eqnarray*}
By \eqref{eqn19990}, \eqref{eqn19989}, \eqref{eqn19970}, \eqref{eqn19969} and \cite[Lemma~6]{LW05},
\begin{eqnarray*}\label{eq0030}
|\beta_{j}|\lesssim \widetilde{\delta}(\sum_{i\not=j}|\beta_{i}|+\|\Psi^{\perp}\|_{H^1(\mathcal{C})})+\|f\|_{H^{-1}(\mathcal{C})}+\sum_{i\not=j}|c_i\langle\Psi_{s_i}^{p-1}\Psi_{s_i}', \Psi_{s_j}\rangle_{L^2(\mathcal{C})}|
\end{eqnarray*}
for all $j=1,2,\cdots,\nu$ and by \eqref{eq0032},
\begin{eqnarray*}\label{eq0031}
\|\Psi^{\perp}\|_{H^1(\mathcal{C})}\lesssim \widetilde{\delta}\sum_{i=1}^{\nu}|\beta_{i}|+\|f\|_{H^{-1}(\mathcal{C})},
\end{eqnarray*}
where $\widetilde{\delta}\to0$ as $\delta\to0$.  It follows from Lemma~\ref{lem0003} and \eqref{eq0033} that \eqref{eq0035} and \eqref{eq0036} hold for $\delta>0$ sufficiently small.
\end{proof}

We are now in the position to prove the following stability.
\begin{proposition}\label{prop0005}
Let $v\in H^1(\mathcal{C})$ be nonnegative such that
\begin{eqnarray*}
(\nu-\frac12)(C_{a,b,N}^{-1})^{\frac{p+1}{p-1}}<\|v\|_{H^1(\mathcal{C})}^2<(\nu+\frac12)(C_{a,b,N}^{-1})^{\frac{p+1}{p-1}}
\end{eqnarray*}
with $\nu\geq2$ and
$\|f\|_{H^{-1}(\mathcal{C})}\leq\delta$ for $\delta>0$ sufficiently small, where $f$ is given by \eqref{eq0060}.  Then either for $b_{FS}(a)\leq b<a+1$ with $a<0$ or for $a\leq b<a+1$ with $a\geq0$ and $a+b>0$, we have
\begin{eqnarray*}
d_*(v)\lesssim \left\{\aligned &\|f\|_{H^{-1}(\mathcal{C})},\quad p>2,\\
&\|f\|_{H^{-1}(\mathcal{C})}|\log(\|f\|_{H^{-1}(\mathcal{C})})|^{\frac12},\quad p=2,\\
&\|f\|_{H^{-1}(\mathcal{C})}^{\frac{p}{2}},\quad 1<p<2.
\endaligned\right.
\end{eqnarray*}
\end{proposition}
\begin{proof}
We recall that $d_*^2(v)=\|\rho\|_{H^1(\mathcal{C})}^2$ and $\rho=\phi+\varphi$.  The conclusion then follows immediately from Lemmas~\ref{lemn0001}, \ref{lem0003} and \ref{lem0004}.
\end{proof}

\section{Optimality of the stability to profile decompositions}
In this section, we will construct examples, as that in \cite{CFM2017,DSW2021}, to show that the orders in Propositions~\ref{propn0001} and \ref{prop0005} are sharp.  Let us begin with the examples for $\nu=1$.
\begin{proposition}\label{propn0002}
Let $b_{FS}(a)\leq b<a+1$ for $a<0$ and $a\leq b<a+1$ for $a\geq0$ and $a+b>0$.  Then the stability stated in Proposition~\ref{propn0001} is sharp in the sense that, there exists nonnegative $v_*\in H^1(\mathcal{C})$, with
\begin{eqnarray*}
\frac12(C_{a,b,N}^{-1})^{\frac{p+1}{p-1}}<\|v_*\|_{H^1(\mathcal{C})}^2<\frac32(C_{a,b,N}^{-1})^{\frac{p+1}{p-1}}\quad\text{and}\quad \|f_*\|_{H^{-1}(\mathcal{C})}\leq\delta
\end{eqnarray*}
for $\delta>0$ sufficiently small, such that $d_0(v_*)\gtrsim\|f_*\|_{H^{-1}(\mathcal{C})}$.
\end{proposition}
\begin{proof}
Let $v_\ve=\Psi+\ve\varphi$ where $\varphi\in C^\infty_0(\mathcal{C})$ is positive and even such that $\langle \Psi', \varphi\rangle_{H^1(\mathcal{C})}=0.$
Then as that in the proof of Proposition~\ref{propn0001}, we have
\begin{eqnarray*}
f_\ve&=&-\Delta_{\theta}v-\partial_t^2v+(a_c-a)^2v-v^p\\
&=&\ve(-\Delta_{\theta}\varphi-\partial_t^2\varphi+(a_c-a)^2\varphi-p\Psi^{p-1}\varphi)+O((\ve\varphi)^{\sigma_p}).
\end{eqnarray*}
It follows that $\|f_\ve\|_{H^{-1}(\mathcal{C})}\lesssim\ve$ for $\ve>0$ sufficiently small.  As that in the proof of Proposition~\ref{propn0001}, it is easy to see that $d_0(v_\ve)\lesssim\ve$ is attained by some $s_\ve\in\bbr$.  Thus, we can rewrite $v_\ve=\Psi_{s_\ve}+\widetilde{\varphi}_\ve$, where $d_0(v_\ve)=\|\widetilde{\varphi}_\ve\|_{H^1(\mathcal{C})}\lesssim\ve$ and
\begin{eqnarray*}
\langle \Psi_{s_\ve}, \widetilde{\varphi}_\ve\rangle_{H^1(\mathcal{C})}=\langle \Psi_{s_\ve}', \widetilde{\varphi}_\ve\rangle_{H^1(\mathcal{C})}=0.
\end{eqnarray*}
Note that
\begin{eqnarray*}
\|\Psi_{s_\ve}-\Psi\|_{H^1(\mathcal{C})}=\|\widetilde{\varphi}_\ve-\ve\varphi\|_{H^1(\mathcal{C})}\lesssim\ve,
\end{eqnarray*}
we have $s_\ve=o_{\ve}(1)$.  Clearly, $\widetilde{\varphi}_\ve$ satisfies
\begin{eqnarray*}
-\Delta_{\theta}\widetilde{\varphi}_\ve-\partial_t^2\widetilde{\varphi}_\ve+(a_c-a)^2\widetilde{\varphi}_\ve=(\Psi_{s_\ve}+\widetilde{\varphi}_\ve)^{p}
-\Psi_{s_\ve}^{p}+f_\ve.
\end{eqnarray*}
Let $\widetilde{f}_\ve$ be the projection of $f_\ve$ in $H^1(\mathcal{C})$, then by the Taylor expansion and some elementary inequalities,
\begin{eqnarray*}
|\langle \widetilde{\varphi}_\ve, \widetilde{f}_\ve\rangle_{H^1(\mathcal{C})}|\gtrsim-|\langle \widetilde{\varphi}_\ve, \widetilde{f}_\ve\rangle_{H^1(\mathcal{C})}|-|\langle \widetilde{\varphi}_\ve^{\sigma_p}, \widetilde{f}_\ve\rangle_{L^2(\mathcal{C})}|+\|f_\ve\|_{H^{-1}}^2.
\end{eqnarray*}
It follows that $d_0(v_\ve)=\|\widetilde{\varphi}_\ve\|_{H^1(\mathcal{C})}\gtrsim\|f_\ve\|_{H^{-1}}$.
\end{proof}

We next construct examples for $\nu\geq2$.
\begin{proposition}\label{prop0006}
Let $b_{FS}(a)\leq b<a+1$ for $a<0$ and $a\leq b<a+1$ for $a\geq0$ and $a+b>0$.  Then the stability stated in Proposition~\ref{prop0005} is sharp in the sense that, there exists nonnegative $v_*\in H^1(\mathcal{C})$ such that \begin{eqnarray*}
d_*(v_*)\gtrsim \left\{\aligned &\|f_*\|_{H^{-1}(\mathcal{C})},\quad p>2,\\
&\|f_*\|_{H^{-1}(\mathcal{C})}|\log(\|f_*\|_{H^{-1}(\mathcal{C})})|^{\frac12},\quad p=2,\\
&\|f_*\|_{H^{-1}(\mathcal{C})}^{\frac{p}{2}},\quad 1<p<2.
\endaligned\right.
\end{eqnarray*}
\end{proposition}
\begin{proof}
Let us consider the following equation:
\begin{eqnarray}\label{eq0045}
\left\{\aligned&-\Delta_{\theta}\widetilde{\phi}_R-\partial_t^2\widetilde{\phi}_R+(a_c-a)^2\widetilde{\phi}_R\\
&=|\sum_{j=1}^{2}\Psi_{s_{j,R}}+\widetilde{\phi}_R|^{p-1}(\sum_{j=1}^{2}\Psi_{s_{j,R}}+\widetilde{\phi}_R)\\
&-\sum_{j=1}^{2}\Psi_{s_{j,R}}^{p}
+\sum_{j=1}^{2}c_{j,R}\Psi_{s_{j,R}}^{p-1}\Psi_{s_{j,R}}',\quad \text{in }\mathcal{C},\\
&\langle \Psi_{s_{j,R}}', \widetilde{\phi}_R\rangle_{H^{1}(\mathcal{C})}=0\quad\text{for all }j=1,2,\endaligned\right.
\end{eqnarray}
where $s_{1,R}=-\frac{R}{2}$ and $s_{2,R}=\frac{R}{2}$.  By Lemma~\ref{lem0003}, \eqref{eq0045} is solvable for $R>0$ sufficiently large with $|c_{1,R}|+|c_{2,R}|\lesssim Q=e^{-(a_c-a)R}$
either for $b_{FS}(a)\leq b<a+1$ with $a<0$ or for $a\leq b<a+1$ with $a\geq0$ and $a+b>0$.
Let $v_R=\sum_{j=1}^{2}\Psi_{s_{j,R}}+\widetilde{\phi}_R$, then
\begin{eqnarray}
f_R&:=&-\Delta_{\theta}v_R-\partial_t^2v_R+(a_c-a)^2v_R-|v_R|^{p-1}v_R\label{eqn4444}\\
&=&\sum_{j=1}^{2}c_{j,R}\Psi_{s_{j,R}}^{p-1}\Psi_{s_{j,R}}'.\notag
\end{eqnarray}
which, together with Lemma~\ref{lemn0001} and Proposition~\ref{prop0005}, implies that
\begin{eqnarray}\label{eq0047}
\|f_R\|_{H^{-1}(\mathcal{C})}\sim\sum_{j=1}^{2}|c_{j,R}|\sim Q
\end{eqnarray}
for $R>0$ sufficiently large.
Note that as that in the proof of Proposition~\ref{prop0003}, we can show that $d_*^2(v_R)\leq\|\widetilde{\phi}_R\|_{H^1(\mathcal{C})}^2$ is attained at $\sum_{j=1}^{2}\Psi_{s'_{j,R}}$ for some $s'_{1,R}$ and $s'_{2,R}$.  Thus, we can rewrite $v_R=\sum_{j=1}^{2}\Psi_{s'_{j,R}}+\widetilde{\varphi}_R$, where $\widetilde{\varphi}_R\in(\text{span}\{\Psi_{s'_{j,R}}'\})^{\perp}$ in $H^1(\mathcal{C})$.  It follows that $d_*^2(v_R)=\|\widetilde{\varphi}_R\|_{H^1(\mathcal{C})}^2\leq\|\widetilde{\phi}_R\|_{H^1(\mathcal{C})}^2$.  Since
\begin{eqnarray*}
\|\sum_{j=1}^{2}\Psi_{s'_{j,R}}-\sum_{j=1}^{2}\Psi_{s_{j,R}}\|_{H^1(\mathcal{C})}\lesssim\|\widetilde{\phi}_R\|_{H^1(\mathcal{C})}\to0\quad\text{as }R\to+\infty,
\end{eqnarray*}
we have $s'_{j,R}=s_{j,R}+o_R(1)$.  Clearly, by \eqref{eq0045}, $\widetilde{\varphi}_R$ satisfies the following equation:
\begin{eqnarray}\label{eq0145}
\left\{\aligned&-\Delta_{\theta}\widetilde{\varphi}_R-\partial_t^2\widetilde{\varphi}_R+(a_c-a)^2\widetilde{\varphi}_R\\
&=|\sum_{j=1}^{2}\Psi_{s'_{j,R}}+\widetilde{\varphi}_R|^{p-1}(\sum_{j=1}^{2}\Psi_{s'_{j,R}}+\widetilde{\varphi}_R)\\
&-\sum_{j=1}^{2}\Psi_{s'_{j,R}}^{p}
+\sum_{j=1}^{2}c_{j,R}\Psi_{s_{j,R}}^{p-1}\Psi_{s_{j,R}}',\quad \text{in }\mathcal{C},\\
&\langle \Psi_{s'_{j,R}}', \widetilde{\varphi}_R\rangle_{H^{1}(\mathcal{C})}=0\quad\text{for all }j=1,2.\endaligned\right.
\end{eqnarray}
Let $\varrho_R:[0, 1]\to[0, 1]$ be a smooth cut-off function such that
\begin{eqnarray}\label{eqn29980}
\varrho_R(t)=\left\{\aligned&1,\quad s'_{1,R}+\frac{R}{2}-3\leq t\leq s'_{1,R}+\frac{R}{2}-2,\\
&0,\quad t\leq s'_{1,R}+\frac{R}{2}-4 \text{ or }t\geq s'_{1,R}+\frac{R}{2}-1.\endaligned\right.
\end{eqnarray}
Then $\|\varrho_R\|_{H^1(\mathcal{C})}\lesssim1$.  Similar to that of \eqref{eqn19985}, \eqref{eqn19984} and \eqref{eqn19999},
\begin{eqnarray*}
|\sum_{j=1}^{2}\Psi_{s'_{j,R}}+\widetilde{\varphi}_R|^{p-1}(\sum_{j=1}^{2}\Psi_{s'_{j,R}}+\widetilde{\varphi}_R)-\sum_{j=1}^{2}\Psi_{s'_{j,R}}^{p}\gtrsim \Psi_{s'_{1,R}}^{p-1}\Psi_{s'_{2,R}}+p\Psi_{s'_{1,R}}^{p-1}\widetilde{\varphi}_R+O(|\widetilde{\varphi}_R|^{\sigma_p})
\end{eqnarray*}
in the region $[s'_{1,R}+\frac{R}{2}-4, s'_{1,R}+\frac{R}{2}-1]$.  Thus, by multiplying \eqref{eq0145} with $\varrho_R$ and integrating by parts,
\begin{eqnarray*}
\|\widetilde{\varphi}_R\|_{H^1(\mathcal{C})}&\gtrsim&\int_{\mathcal{C}}(\Psi_{s'_{1,R}}^{p-1}\Psi_{s'_{2,R}}+p\Psi_{s'_{1,R}}^{p-1}\widetilde{\varphi}_R
+O(|\widetilde{\varphi}_R|^{\sigma_p}))\varrho_R
+\int_{\mathcal{C}}\sum_{j=1}^{2}c_{j,R}\Psi_{s_{j,R}}^{p-1}\Psi_{s_{j,R}}'\varrho_R\\
&\gtrsim&\int_{\mathcal{C}}\Psi_{s'_{1,R}}^{p-1}\Psi_{s'_{2,R}}\varrho_R-\sum_{j=1}^{2}|c_{j,R}|\int_{\mathcal{C}}\Psi_{s_{j,R}}^{p-1}\Psi_{s_{j,R}}'\varrho_R
-\|\widetilde{\varphi}_R\|_{H^1(\mathcal{C})}.
\end{eqnarray*}
By \eqref{eq0026} and \eqref{eqn29980},
\begin{eqnarray*}
\sum_{j=1}^{2}|c_{j,R}|\int_{\mathcal{C}}\Psi_{s_{j,R}}^{p-1}\Psi_{s_{j,R}}'\varrho_R=o(\sum_{j=1}^{2}|c_{j,R}|)\quad\text{as }R\to+\infty
\end{eqnarray*}
and
\begin{eqnarray*}
\int_{\mathcal{C}}\Psi_{s'_{1,R}}^{p-1}\Psi_{s'_{2,R}}\varrho_R\gtrsim\int_{\frac{R}{2}-3}^{\frac{R}{2}-2}e^{-(p-1)(a_c-a)r}e^{-(a_c-a)(R-r)}\sim\left\{\aligned&Q,\quad p\geq2,\\
&Q^{\frac{p}{2}},\quad 1<p<2\endaligned\right.
\end{eqnarray*}
for $R>0$ sufficiently large.
It follows from \eqref{eq0047} that
\begin{eqnarray*}
d_*(v_R)\gtrsim \left\{\aligned &\|f_R\|_{H^{-1}(\mathcal{C})},\quad p>2,\\
&\|f_R\|_{H^{-1}(\mathcal{C})}^{\frac{p}{2}},\quad 1<p<2
\endaligned\right.
\end{eqnarray*}
for $R>0$ sufficiently large.
For $p=2$, we shall modify the test function $\varrho_R$ by
\begin{eqnarray}\label{eqn19980}
\widetilde{\varrho}_R(t)=\left\{\aligned&1,\quad s'_{1,R}+\frac{R}{4}\leq t\leq s'_{1,R}+\frac{R}{2}-2,\\
&0,\quad t\leq s'_{1,R}+\frac{R}{4}-1 \text{ or }t\geq s'_{1,R}+\frac{R}{2}-1.\endaligned\right.
\end{eqnarray}
Then $\|\widetilde{\varrho}_R\|_{H^1(\mathcal{C})}\lesssim\sqrt{R}$ for $R>0$ sufficiently large.  Thus, by multiplying \eqref{eq0145} with $\widetilde{\varrho}_R$ and integrating by parts,
\begin{eqnarray*}
\sqrt{R}\|\widetilde{\varphi}_R\|_{H^1(\mathcal{C})}&\gtrsim&\int_{\mathcal{C}}(\Psi_{s'_{1,R}}\Psi_{s'_{2,R}}+p\Psi_{s'_{1,R}}\widetilde{\varphi}_R
+O(\widetilde{\varphi}_R^{2}))\widetilde{\varrho}_R
+\int_{\mathcal{C}}\sum_{j=1}^{2}c_{j,R}\Psi_{s_{j,R}}\Psi_{s_{j,R}}'\widetilde{\varrho}_R\\
&\gtrsim&\int_{\mathcal{C}}\Psi_{s'_{1,R}}\Psi_{s'_{2,R}}\widetilde{\varrho}_R-\sum_{j=1}^{2}|c_{j,R}|\int_{\mathcal{C}}\Psi_{s_{j,R}}\Psi_{s_{j,R}}'\widetilde{\varrho}_R
-\|\widetilde{\varphi}_R\|_{H^1(\mathcal{C})}\sqrt{R}.
\end{eqnarray*}
By \eqref{eq0026} and \eqref{eqn19980},
\begin{eqnarray*}
\sum_{j=1}^{2}|c_{j,R}|\int_{\mathcal{C}}\Psi_{s_{j,R}}\Psi_{s_{j,R}}'\widetilde{\varrho}_R\sim
\sum_{j=1}^{2}|c_{j,R}|\int_{\frac{R}{4}}^{\frac{R}{2}}e^{-2(a_c-a)r}\sim\sum_{j=1}^{2}|c_{j,R}| e^{-(a_c-a)\frac{R}{2}}
\end{eqnarray*}
as $R\to+\infty$.
On the other hand,
\begin{eqnarray*}
\int_{\mathcal{C}}\Psi_{s'_{1,R}}\Psi_{s'_{2,R}}\widetilde{\varrho}_R\gtrsim\int_{\frac{R}{4}}^{\frac{R}{2}-1}e^{-(a_c-a)r}e^{-(a_c-a)(R-r)}\sim Re^{-(a_c-a)R}
\end{eqnarray*}
for $R>0$ sufficiently large.
It follows from \eqref{eq0047} that
\begin{eqnarray*}
d_*(v_R)\gtrsim\|f_R\|_{H^{-1}(\mathcal{C})}|\log(\|f_R\|_{H^{-1}(\mathcal{C})})|^{\frac12}
\end{eqnarray*}
for $p=2$ and $R>0$ sufficiently large.  Now, we take $v_*=v_R^+$ then $v_*=v_R+v_R^-$, where $v_R^\pm=\max\{\pm v_R, 0\}$.  Clearly, we have $0\leq v_R^-\leq |\widetilde{\phi}_R|$ since $\sum_{j=1}^{2}\Psi_{s_{j,R}}$ is positive.  It follows from \eqref{eqn4444} that
\begin{eqnarray*}
\|v_R^-\|_{H^1(\mathcal{C})}^2=\langle v_R^-, V_R\rangle_{H^1(\mathcal{C})}=\int_{\mathcal{C}}|v_R|^{p-1}v_Rv_R^-+\int_{\mathcal{C}}f_Rv_R^-\lesssim\|v_R^-\|_{H^1(\mathcal{C})}^{p+1}+\int_{\mathcal{C}}|f_R||v_R^-|.
\end{eqnarray*}
For $1<p\leq 2$, by Lemma~\ref{lem0003} and \eqref{eq0047},
\begin{eqnarray*}
\|v_R^-\|_{H^1(\mathcal{C})}^2\lesssim\int_{\mathcal{C}}|f_R||\widetilde{\phi}_R|\lesssim\left\{\aligned &Q^2|\log(Q)|^{\frac12},\quad p=2,\\
&Q^{1+\frac{p}{2}},\quad 1<p<2.
\endaligned\right.
\end{eqnarray*}
For $p>2$, recall that by Lemma~\ref{lem0003}, $\|\widetilde{\phi}_R\|_{\sharp}\lesssim1$.  Thus, by \eqref{eq0026} and \eqref{eqn19996}, $v_R^-=0$ for $|t-s_{j,R}|\leq \frac{R}{2}$.  It follows that
\begin{eqnarray*}
\|v_R^-\|_{H^1(\mathcal{C})}^2\lesssim\int_{\mathcal{C}}|f_R||v_R^-|\lesssim\int_{(\cup_{j=1}^2\{|t-s_{j,R}|\leq R\})^c}|f_R||\widetilde{\phi}_R|=o(Q^2).
\end{eqnarray*}
The conclusion then follows from $d_*(v_*)\geq d_*(v_R)-\|v_R^-\|_{H^1(\mathcal{C})}$.
\end{proof}

We close this section by the proof of Theorem~\ref{thm0002}.

\vskip0.12in

\noindent\textbf{Proof of Theorem~\ref{thm0002}:} It follows immediately from \eqref{eq0007} and Propositions~\ref{propn0001}, \ref{prop0005}, \ref{propn0002} and \ref{prop0006}.
\hfill$\Box$

\section{Acknowledgements}
The research of J. Wei is
partially supported by NSERC of Canada and the research of Y. Wu is supported by NSFC (No. 11971339).


\begin{thebibliography}{100}
\bibitem{A1976}
T. Aubin, Probl\`emes isop\'erim\'etriques de Sobolev.  {\it J. Differential Geometry,} {\bf11} (1976), 573--598.

\bibitem{BE1991}
G. Bianchi, H. Egnell, A note on the Sobolev inequality.  {\it J. Funct. Anal.,} {\bf100} (1991), 18--24.

\bibitem{BL1985}
H. Brezis, E. Lieb, Sobolev inequalities with remainder terms. {\it J. Funct. Anal.,} {\bf62} (1985), 73--86.

\bibitem{CKN1984}
L. Caffarelli, R. Kohn, L. Nirenberg, First order interpolation inequalities with weights.  {\it Compos. Math.,} {\bf53} (1984), 259--275.

\bibitem{CW2001}
F. Catrina, Z.-Q. Wang, On the Caffarelli-Kohn-Nirenberg inequalities: sharp constants, existence (and nonexistence), and symmetry of extremal functions. {\it Comm. Pure Appl. Math.,} {\bf54} (2001), 229--258.

\bibitem{CC1993}
K. Chou, W. Chu, On the best constant for a weighted Sobolev-Hardy inequality.  {\it J. London Math. Soc.,} {\bf48} (1993), 137--151.

\bibitem{CFM2017}
G. Ciraolo, A. Figalli, F. Maggi, A quantitative analysis of metrics on $\bbr^N$ with
almost constant positive scalar curvature, with applications to fast diffusion flows.  {\it Int. Math. Res. Not.,} {\bf2018} (2017), 6780--6797.

\bibitem{PDM2003}
M. Del Pino, J. Dolbeault, M. Musso, ``Bubble-tower'' radial solutions
in the slightly supercritical Brezis-Nirenberg problem. {\it J. Differential Equations,} {\bf193} (2003), 280--306.

\bibitem{PFM2003}
M. del Pino, P. Felmer, M. Musso. Two-bubble solutions in the super-critical Bahri-Coron's problem.  {\it Calc. Var. PDEs,} {\bf16} (2003), 113--145.

\bibitem{PMPP2011}
M. del Pino, M. Musso, F. Pacard, A. Pistoia. Large energy
entire solutions for the Yamabe equation. {\it J. Differential Equations,} {\bf251} (2011), 2568--2597.

\bibitem{DSW2021}
B. Deng, L. Sun, J. Wei, Optimal quantitative estimates of Struew's decomposition.  {\it Preprint,} arXiv:2103.15360v1  [math.AP].

\bibitem{DELT2009}
J. Dolbeault, M. J. Esteban, M. Loss, and G. Tarantello, On the symmetry of extremals for the Caffarelli-Kohn-Nirenberg inequalities.  {\it Adv. Nonlinear Stud.,} {\bf9} (2009), 713--726.

\bibitem{DEL2012}
J. Dolbeault, M. J. Esteban, M. Loss, Symmetry of extremals of functional inequalities via spectral estimates for linear operators.  {\it J. Math. Phys.,} {\bf53} (2012), article 095204, 18 pp.

\bibitem{DEL2016}
J. Dolbeault, M. J. Esteban, M. Loss, Rigidity versus symmetry breaking via nonlinear flows on cylinders and Euclidean spaces, {\it Invent. math.,} {\bf206} (2016), 397--440.

\bibitem{FG2021}
A. Figalli, F. Glaudo, On the Sharp Stability of Critical Points of the Sobolev Inequality.  {\it Arch. Rational Mech. Anal.,} {\bf237} (2020), 201--258.

\bibitem{FS2003}
V. Felli, M. Schneider, Perturbation results of critical elliptic equations of Caffarelli-Kohn-Nirenberg type. {\it J. Differential Equations,} {\bf191} (2003), 121--142.

\bibitem{L1983}
E. Lieb, Sharp constants in the Hardy-Littlewood-Sobolev and related inequalities. {\it Ann. of Math. (2),} {\bf118} (1983), 349--374.

\bibitem{LW2004}
C.-S. Lin, Z.-Q. Wang, Symmetry of extremal functions for the Caffarrelli-Kohn-Nirenberg inequalities. {\it Proc. Amer. Math. Soc.,} {\bf132} (2004), 1685--1691.

\bibitem{LW05}
T.-C. Lin, J. Wei, Ground state of $N$ coupled nonlinear Schr\"odinger equations in $\bbr^n$, $n\leq3$, {\it Comm. Math. Phys.,}
{\bf255} (2005),  629-653.

\bibitem{RSW2002}
V. Radulescu, D. Smets, M. Willem, Hardy-Sobolev inequalities with remainder terms.  {\it Topol. Methods Nonlinear Anal.,} {\bf20} (2002), 145--149.

\bibitem{S1984}
M. Struwe, A global compactness result for elliptic boundary value problems involving limiting nonlinearities.  {\it Math. Z.,} {\bf187} (1984), 511--517.

\bibitem{S2000}
M. Struwe, Variational methods. Applications to Nonlinear Partial Differential
Equations and Hamiltonian Systems. Third edition. Springer-Verlag, Berlin, 2000. xviii+274 pp.

\bibitem{T1976}
G. Talenti, Best constant in Sobolev inequality.  {\it Ann. Mat. Pura Appl. (4),} {\bf110} (1976), 353--372.

\bibitem{WW2000}
Z.-Q. Wang, M. Willem, Singular Minimization Problems, {\it J. Differential Equations,} {\bf161} (2000), 307--320.


\bibitem{WW2020}
J. Wei, Y. Wu, Ground states of nonlinear elliptic systems with mixed couplings, {\it J. Math. Pures Appl.,} {\bf141} (2020), 50--88.

\bibitem{WY2007}
J. Wei, S. Yan, Arbitrary many boundary peak solutions for an elliptic Neumann problem with critical growth. {\it J. Math. Pures Appl.,} {\bf88} (2007), 350--378.

\bibitem{WY2010}
J. Wei, S. Yan, Infinitely many solutions for the prescribed scalar curvature problem on $\mathbb{S}^N$. {\it J. Funct. Anal.,} {\bf258} (2010), 3048--3081.
\end{thebibliography}
\end{document}